\documentclass[11pt, a4paper]{article}
\usepackage[latin1]{inputenc}
\usepackage[T1]{fontenc} 
\usepackage{amsmath}
\usepackage{amsfonts}
\usepackage{amssymb}
\usepackage{amsthm}
\usepackage[affil-it]{authblk}
\usepackage{enumitem}
\usepackage[normalem]{ulem}
\usepackage[mathscr]{euscript}
\usepackage{dsfont}
\usepackage{float}
\usepackage{times}
\usepackage{bbm}
\usepackage{xcolor}
\usepackage{xparse}
\usepackage{graphicx}
\usepackage[hang]{footmisc}

\usepackage[colorlinks,linkcolor=blue,citecolor=blue,urlcolor=black]{hyperref}

\usepackage[capitalise]{cleveref}
\usepackage[left=3cm, right=3cm, bottom=3cm, top=3cm]{geometry}

\newtheorem{theorem}{Theorem}[section]
\newtheorem{lemma}[theorem]{Lemma}
\newtheorem{proposition}[theorem]{Proposition}
\newtheorem{corollary}[theorem]{Corollary}

\theoremstyle{definition}
\newtheorem{definition}[theorem]{Definition}

\newtheorem{condition}[theorem]{Condition}

\theoremstyle{remark}
\newtheorem{remark}[theorem]{Remark}
\newtheorem{example}[theorem]{Example}

\numberwithin{equation}{section}

\crefname{example}{Example}{Examples}
\Crefname{example}{Example}{Examples}

\crefname{assumption}{Assumption}{Assumptions}
\Crefname{assumption}{Assumption}{Assumptions}

\crefname{condition}{Condition}{Conditions}
\Crefname{condition}{Condition}{Conditions}

\let\para\S

\makeatletter
\newcommand{\crefnames}[3]{%
  \@for\next:=#1\do{%
    \expandafter\crefname\expandafter{\next}{#2}{#3}%
  }%
}
\makeatother

\crefnames{part,chapter,section}{\para}{\para\para}

\setlist{topsep=1ex, itemsep=0.5ex, before={\setlist{topsep=-.5ex}}}

\DeclareMathOperator{\id}{id}
\DeclareMathOperator{\ran}{ran}

\DeclareMathOperator{\supp}{supp}

\renewcommand{\d}{\ensuremath{\mathfrak{d}}}
\newcommand{\f}{\ensuremath{\frac}}
\newcommand{\fG}{\ensuremath{\mathfrak{G}}}

\newcommand{\A}{\ensuremath{\mathcal{A}}}
\newcommand{\B}{\ensuremath{\mathcal{B}}}
\newcommand{\C}{\ensuremath{\mathcal{C}}}
\newcommand{\cE}{\ensuremath{\mathcal{E}}}

\newcommand{\bD}{\ensuremath{\mathbb{D}}}
\newcommand{\D}{\ensuremath{\mathcal{D}}}
\renewcommand{\S}{{\mathcal{S}}}

\newcommand{\F}{\ensuremath{\mathcal{F}}}

\renewcommand{\H}{\ensuremath{\mathscr{H}}}
\newcommand{\cH}{\ensuremath{\mathcal{H}}}
\newcommand{\cK}{\ensuremath{\mathcal{K}}}
\newcommand{\sC}{\ensuremath{\mathscr{C}}}
\newcommand{\sP}{\ensuremath{\mathscr{P}}}
\newcommand{\T}{\ensuremath{\mathbb{T}}}

\newcommand{\I}{\ensuremath{\mathcal{I}}}

\newcommand{\fS}{\ensuremath{\mathfrak{S}}}

\newcommand{\h}{\ensuremath{\mathfrak{h}}}
\newcommand{\K}{\ensuremath{\mathfrak{K}}}
\renewcommand{\L}{\ensuremath{\mathcal{L}}}
\newcommand{\M}{\ensuremath{\mathcal{M}}}
\newcommand{\N}{\ensuremath{\mathbb{N}}}
\renewcommand{\P}{\ensuremath{\mathcal{P}}}
\newcommand{\R}{\ensuremath{\mathbb{R}}}
\newcommand{\X}{\ensuremath{\mathcal{X}}}
\newcommand{\1}{\ensuremath{\mathds{1}}
}\newcommand{\Y}{\ensuremath{\mathcal{Y}}}

\def\CY{\mathcal Y}

\def\<{\langle}
\def\>{\rangle}

\def\I{\mathcal I}

\NewDocumentCommand{\Lin}{om}{\IfNoValueTF{#1}{L(\R^{#2},\R^{#2})}{L(\R^{#1},\R^{#2})}}
\NewDocumentCommand{\Cb}{om}{\IfNoValueTF{#1}{\C_b^{#2}}{\C_b^{#2,#1}}}

\def\le{\leq}
\NewDocumentCommand{\Lip}{om}{\IfNoValueTF{#1}{|#2|_{\mathrm{Lip}}}{|#2|_{\mathrm{Lip};\,#1}}}

\newcommand{\define}{\ensuremath\triangleq}
\newcommand{\expec}[1]{\mathbb{E}[#1]}
\def\E{\mathbb E}
\newcommand{\Expec}[1]{\mathbb{E}\left[#1\right]}

\newcommand{\braket}[1]{\ensuremath\langle#1\rangle}
\newcommand{\Braket}[1]{\ensuremath\left\langle#1\right\rangle}
\newcommand{\prob}{\ensuremath\mathbb{P}}

\newcommand{\TV}[1]{\ensuremath{\left\|#1\right\|_\mathrm{TV}}}

\renewcommand{\geq}{\geqslant}
\renewcommand{\leq}{\leqslant}

\def\${|\!|\!|}
\def\B{{\mathcal B}}
\def\F{{\mathcal F}}
\def\ge{\geq}
\def\1{\mathbf 1}
\newcommand{\vertiii}[1]{{\left\vert\kern-0.25ex\left\vert\kern-0.25ex\left\vert #1 \right\vert\kern-0.25ex\right\vert\kern-0.25ex\right\vert}}
\newcommand{\bvertiii}[1]{{\big\vert\kern-0.25ex\big\vert\kern-0.25ex\big\vert #1 \big\vert\kern-0.25ex\big\vert\kern-0.25ex\big\vert}}
\newcommand{\rom}[1]{(\textup{\uppercase\expandafter{\romannumeral#1}})}

\makeatletter
\newcommand{\substackal}[1]{%
  \vcenter{%
    \Let@ \restore@math@cr \default@tag
    \baselineskip\fontdimen10 \scriptfont\tw@
    \advance\baselineskip\fontdimen12 \scriptfont\tw@
    \lineskip\thr@@\fontdimen8 \scriptfont\thr@@
    \lineskiplimit\lineskip
    \ialign{\hfil$\m@th\scriptstyle##$&$\m@th\scriptstyle{}##$\hfil\crcr
      #1\crcr
    }%
  }%
}
\makeatother

\newcommand{\bigcdot}{\boldsymbol{\cdot}}

\definecolor{LB}{rgb}{0.29, 0.63, 0.73}

\begin{document}
\title{Mild Stochastic Sewing Lemma, SPDE in Random Environment, and Fractional Averaging}
\author{Xue-Mei Li\thanks{\href{mailto:xue-mei.li@imperial.ac.uk}{xue-mei.li@imperial.ac.uk}. Research  support by the EPSRC (grants nos:  EP/S023925/1 and EP/V026100/1).} }
\author{Julian Sieber\thanks{\href{mailto:j.sieber19@imperial.ac.uk}{j.sieber19@imperial.ac.uk}.}}
\affil{Department of Mathematics, Imperial College London, UK }
\date{\today}
\maketitle

\begin{abstract}
  \noindent
 Our first result is a stochastic sewing lemma with quantitative estimates for  mild incremental processes, with which we study  SPDEs driven by fractional Brownian motions in a random environment. We obtain uniform $L^p$-bounds. Our second result is a fractional averaging principle admitting non-stationary fast environments. As an application,  we prove a fractional averaging principle for SPDEs.
  \vspace{0.2cm}
\end{abstract}

{\hypersetup{hidelinks}
\setcounter{tocdepth}{2} 
\tableofcontents
}

\newpage
\section{Introduction}

{\bf A.}  This article is  written on the occasion of Bj\"orn Schmalfu{\ss}' 65th birthday. There had been  significant novel developments both on stochastic analysis of and on stochastic dynamics driven by fractional Brownian motions---these topics are at the heart of Bj\"orn Schmalfu{\ss}'s research. Instead of  reporting on recent progresses, this article contains  original research 
 which deserves both a careful  treatment and which also illustrates recent directions in two-scale random dynamics driven non-Markovian processes.
We propose to  study stochastic partial differential equations (SPDEs) in a random environment. The driving noise $B$ is a fractional Brownian motion (fBm) with trace-class covariance operator $Q$, which takes values in a separable Hilbert space $\cK$ and is independent of the fast environment. 
The  Hurst parameter $H$ of $B$ is assumed to be greater than $\f 12$. Fractional noise with large Hurst parameter 
models noise increments with long range dependence (LRD). The LRD phenomenon is prevalent in data from social economical and natural sciences.

The solution of the SPDE takes values in a separable Hilbert space $\cH$ and the environment process takes its value in a complete separable metric space $\Y$. Let $f: \cH\times\Y\to\cH$, $g: \cH\times\Y\to\L(\cK,\cH)$, and $A:\cH\to\cH$ a densely defined closed linear operator with $\D(A)$ generating a bounded analytical semigroup $(S_t)_{t\geq 0}$. We consider
\begin{equation}\label{eq:intro_spde}
  dX_t= AX_t\,dt+f\big(X_t, Y_t\big)\,dt+g\big(X_t, Y_t\big) \,dB_t.
\end{equation}
The integral in the SPDE is interpreted as Riemann-Stieltjes integrals for each fixed sample path $B(\omega)$. We say that a path $X:[0,T]\to\cH$ is a \emph{mild solution} to \eqref{eq:intro_spde} if, for any $t\leq T$,
  \begin{equation}\label{def:mild_solution}
     X_t(\omega)=S_tX_0(\omega)+\int_0^tS_{t-s}f\big(X_s(\omega), Y_s(\omega)\big)\,ds+\int_0^t S_{t-s}g\big(X_s(\omega), Y_s(\omega)\big)\,dB_s(\omega)\qquad\prob\text{-a.s.}
  \end{equation} 
Preliminary results on the pathwise calculus are presented in \cref{pathwise}.

{\bf B.} Our first technical result  is  a `mild' version of the stochastic sewing lemma. This lemma, presented in \cref{prop:stochastic_mild_sewing},  generalizes the sewing lemma of Gubinelli, Feyel, and de La Pradelle \cite{Gubinelli-lemma,Feyel-delaPradelle}, the stochastic sewing lemma of L\^e \cite{Le2020}, and the mild sewing lemma of Gubinelli and Tindel  \cite[Theorem 3.5]{Gubinelli2010}, see also Gerasimovi\v{c}s and Hairer \cite{Gerasimovics2019}. For the mild stochastic sewing lemma the incremental process takes into account the interaction between the linear motion of $A$ and the nonlinear part.  The  mild stochastic sewing lemma is convenient for studying mild solutions of SPDEs, its first application  is a uniform $L^p$-bounds for the solution of the SPDE in a class of random environments. Stochastic sewing lemma in Banach spaces were recently obtained by L\^e \cite{Le-sewing-banach-21} and by
Athreya,  Butkovsky, L\^e, and Mytnik \cite{Athreya-Le-Butkovsky-Mytnik}, albeit both with the usual increment process. After submitting this to the arXiv, we learnt from L\^e a nice trick allowing to obtain the sewing part of the lemma from that in \cite{Le-sewing-banach-21}, while the error estimate  does not  follow directly from the corresponding result in their latest article  \cite{Athreya-Le-Butkovsky-Mytnik}.

We then turn to the main topic of this article:  stochastic evolution equations in a random environment evolving at a faster scale.
In a stochastic slow-fast model, the slow variables evolve in the random environment modelled by the fast random motion. The principle of fractional averaging leverages ergodicity of the fast variable to obtain an approximation of the slow variables $(X_t^\varepsilon)_{t\in[0,T]}$ as $\varepsilon\to 0$ whose microscopic structure resembles a fractional noise $B$. Despite of the statistical evidence, such slow-fast stochastic models have been overlooked because of the lack of tools for studying them. Stochastic averaging for noise with LRD has attracted a lot of attention recently, which has led to renovation of both techniques and results.  
Non-Markovian multi-scale dynamics only very recently gained traction, see e.g. \cite{Hairer2020,  Gehringer2021, Gehringer-Li-2020-1, Gehringer-Li-2020, Eichinger2020, Bourguin-Gailus-Spiliopoulos,brehier-numerics, Pei2020,Pei-Xu-Yin-Zhang, Sun2021}.  Pathwise well posedness of the equations and bounds on its solutions are given in \Cref{pathwise}; uniform $L^p$ bounds  of solutions in $\epsilon$  are obtained with the mild stochastic sewing lemma  in the subsequent sections.

{\bf C.} Our next technical result is an ergodic theorem allowing to obtain a fractional averaging principle for non-stationary and possibly non-Markovian environment.
Suppose that the random environment evolves on the \emph{macroscopic} time-scale $\varepsilon^{-1}$ and converges to a stationary distribution $\pi$ with algebraic rate $t^{-\delta}$ and $f,g$ satisfy suitable analytic conditions. One defines  $$\bar f(x)\define\int_\Y f(x,y)\pi(dy), \qquad \bar g(x)\define\int_{\Y} g(x,y) \pi(dy).$$
The key to the averaging theorem is to show that
 $|f(\cdot, Y_{\frac{\cdot}{\varepsilon}}) -\bar{f}|_{{-\kappa,\gamma}}\to 0$ (and similarly for $g$), in a H\"older norm with negative exponent $-\kappa<0$ in time.
 This was shown  in \cite{Hairer2020} for a mixing stationary environment and both $\cH,\cK$ finite-dimensional Euclidean spaces. In \cref{sec-ergodic}, we obtain an ergodic theorem for non-stationary fast processes allowing also correlation of the increments:
 \begin{align*}\label{ergodic-theorem}
       \Big\|\Big|g_\varepsilon(\bigcdot,x)-g_\varepsilon(\bigcdot,z)-\big(\bar g(x)-\bar g(z)\big)\Big|_{\C^{-\kappa}(L(\cK,\cH))}\Big\|_{L^p}&\lesssim \varepsilon^{\frac\delta p}|x-z|, \\
       \Big\|\big|g_\varepsilon(\bigcdot,x)-\bar g(x)\big|_{\C^{-\kappa}(L(\cK,\cH))}\Big\|_{L^p}&\lesssim \varepsilon^{\frac\delta p},
  \end{align*}
where $g_\epsilon(t, x)=g(x, Y_{\f t\epsilon})$. See \cref{lem:norm_convergence} and the subsequent examples. [These are the ergodic conditions needed for  fractional averaging, see \cref{def-ergodic} and \cref{thm:averaging_spde} for details.] The ergodic theorem nicely illustrates the fundamental differences between Markovian and non-Markovian dynamics. We show that the ergodic theorem actually follows from the total variation convergence of the conditioned dynamics, which can often be easily verified for Markovian fast processes. If $Y$ is also the solution to a fractional stochastic equation, its disintegration is usually governed by a different dynamics and no longer solves the original equation, see \cite{Li2022,Li2020} for a detailed discussion.

{\bf D.} As an application of the  ergodic theorem  and the mild stochastic sewing lemma, we obtain  a fractional averaging principle for a two-scale SPDE:  Suppose that $Y$  has the $\C^{\f 12-}$ regularity of a typical diffusion process and consider the SPDE
\begin{equation}\label{eq:slow_spde}
  dX_t^\varepsilon= AX_t^\varepsilon\,dt+f\big(X_t^\varepsilon, Y_{\frac{t}{\varepsilon}}\big)\,dt+g\big(X_t^\varepsilon, Y_{\frac{t}{\varepsilon}}\big) \,dB_t,\qquad X_0^\varepsilon=X_0.
\end{equation}
As $\varepsilon\to 0$ we expect $X^\varepsilon$ to be well approximated by an \emph{effective} autonomous dynamics $\bar{X}$ obtained through averaging of the coefficients. These kind of problem have a long history in the theory of dynamical systems and stochastic equations; in fact stochastic averaging for Markovian systems came up shortly after stochastic integration was introduced \cite{Bogolyubov1955,Hasminskii1968}. The typical questions are: What is the nature of the effective dynamics  and what is the mode of convergence?  

These questions are answered by our fractional averaging principle presented in \cref{thm:averaging_spde}. It is obtained under the ergodic assumptions on $f,g$ stated in paragraph C as well as additional regularity conditions ensuring the well-posedness of the equation and the uniform bounds on the solution, discussed in paragraph B. Examples of fast motions with the given ergodic assumptions include Markov processes with polynomial rate of convergence to their stationary measures, stationary processes with an algebraic mixing rate, and fractional dynamics satisfying a large scale contraction assumption. As $\varepsilon\to 0$, we prove that the pathwise mild solution $X^\varepsilon$ to \eqref{eq:slow_spde} converges in (mild) H\"older norm in probability to the solution of the na\"ively averaged, autonomous equation
\begin{equation}\label{eq:averaged_spde}
  d\bar{X}_t=A\bar{X}_t+\bar{f}(\bar{X}_t)\,dt+\bar{g}(\bar{X}_t)\,dB_t,\qquad \bar{X}_0=X_0.
\end{equation}

{\bf E.} To the best of our knowledge, this article proves the first genuine fractional averaging theorem for SPDEs.
The difference between the usual (Markovian) and the fractional effective dynamics lies in the definition of $\bar g$. For the former, $\bar g$ is a square root of $g(x,y)g^T(x,y)$, and---in case $\bar g$ does not vanish---the convergence is in a weak sense. If $g(x,y)$ does not depend on $y$, the fractional and the Markovian average of $g$ agree, whence classical methods apply.   For SPDEs driven by fractional Brownian motions  this was exploited in \cite{Pei2020,Pei-Xu-Yin-Zhang, Sun2021},  where again $g(x,y)=g(x)$ and classical averaging principle holds.  The convergence obtained in our averaging theorem is in probability,  which does not hold for SPDE/SDE driven by a Wiener processes, to the latter there is a  vast readily available literature which we do not attempt to enumerate.  

{\it Acknowledgement. The authors are grateful to Khoa L\^e for a very helpful conversation. We also thank two anonymous referees and Alexandra Neam\c{t}u for their comments.}
\section{Pathwise Solutions}\label{pathwise}

In this section we develop the mild solution theory  for the evolution equation
\begin{equation*}
  dx_t=\big(Ax_t+f(t,x_t)\big)\,dt+g(t,x_t)\,d\h_t,
\end{equation*}
where $\h\in\C^\gamma\big([0,T],\cK\big)$, $\gamma>\frac12$, is a deterministic H\"older continuous path, and  obtain pathwise bounds.

\subsection{Preliminaries}
The aim of the section is to fix notation and conventions, they are essentially standard. 
Throughout this article $\cH$ is a Hilbert space. Let $A$ be the generator of a strongly continuous semigroup $(S_t)_{t\geq 0}$. Although multiple norms are needed in the sequel, the symbol $\|\cdot\|$ will always denote the norm $\cH$ and $\|S_t\|$ the norm of $S_t$ regarded as a linear map $\cH\to\cH$.

Recall that a strongly continuous semigroup $(S_t)_{t\geq 0}$ on $\cH$ is an \emph{analytic semigroup} of angle $\theta \in (0, \f \pi 2]$ if
   $(S_t)_{t\geq 0}$ extends to 
$$ \Delta_\theta:=\{z\in\mathbb{C}\setminus \{0\}:\,|\arg z|<\theta\}$$
 such that
 (i) $S_{z_1+z_2}=S_{z_1}S_{z_2}$ for $ z_1,z_2\in\Delta_\theta$; 
(ii) the map $z\mapsto S_z$ is analytic on  $ \Delta_ \theta$, and (iii)
for any $\theta'<\theta$ and every $x\in \cH$, the strong continuity $\lim_{z\to 0} S_{z_1}x=x$ holds along any sequence in $\Delta_{\theta'}$. 
If in addition, $\sup_{z \in \Delta_{\theta'}} \|S_z\|<\infty$ for every $\theta<\theta$ we say $S(t)$ is a bounded analytic semigroup.
We assume this condition on $S_t$, equivalently $A$ is sectorial.

Without loss of generality,  by shifting $A$ if needed, we assume that  there are constant $C,\nu>0$ such that 
\begin{equation}\label{eq:semigroup_decay}
  \|S_t\|\leq C e^{-\nu t}\qquad\forall\,t\geq 0.
\end{equation}
By standard theory \cite{Pazy1974}, this allows us to define the bounded injective operator
\begin{equation*}
  (-A)^{-\kappa}\define\frac{1}{\Gamma(\kappa)}\int_0^\infty t^{\kappa-1}S_t\,dt,\qquad\kappa>0.
\end{equation*}
Here $\Gamma(\kappa)\define\int_0^\infty t^{\kappa-1}e^{-t}\,dt$ is Euler's gamma function.  Then $\cH$ is a pre-Hilbert space with respect to the new norm 
$\|(-A)^\kappa \bigcdot\|_\cH$ whose completion we denote by $\cH_{-\kappa}$. We also set $A^0=I$ and  $(-A)^{\kappa}=\big((-A)^{-\kappa}\big)^{-1}$. In the latter case, $(-A)^{\kappa}$ is a closed densely defined operator  with domain  $\ran \big((-A)^{-\kappa}\big)$. If $\kappa_2>\kappa_1>0$, then 
$\D\big((-A)^{-\kappa_2}\big)\subset \D\big((-A)^{-\kappa_1}\big)$. In summary, we have defined the following interpolation spaces 
\begin{equation}\label{eq:interpolation_spaces}
  \cH_\kappa\define\begin{cases}\Big(\D\big((-A)^\kappa\big),\|\bigcdot\|_{\cH_\kappa}=\|(-A)^\kappa \bigcdot\|_\cH\Big) & \kappa>0,\\
  \Big(\cH,\|\bigcdot\|_{\cH}\Big) & \kappa=0,\\
  \Big(\overline{\cH}^{\|\bigcdot\|_{\cH_\kappa}},\|\bigcdot\|_{\cH_\kappa}=\|(-A)^\kappa \bigcdot\|_\cH\Big) & \kappa<0,
  \end{cases}
\end{equation}
which are themselves Hilbert spaces. In case $A$ is self-adjoint and negative-definite,  then for any $\kappa\in\R$, $A^\kappa$ coincides with the operator obtained by functional calculus.

There are the following analytical bounds.
 If $\kappa_1\leq\kappa_2$, then $\cH_{\kappa_2}\hookrightarrow\cH_{\kappa_1}$ densely and
\begin{equation}\label{eq:holder_semigroup}
  \|(\id-S_t)x\|_{\cH_{\kappa_1}}\lesssim t^{\kappa_2-\kappa_1}\|x\|_{\cH_{\kappa_2}},
\end{equation}
uniformly in $t\geq 0$ for $x\in \cH_{\kappa_2}$, provided that $\kappa_2-\kappa_1\in(0,1)$. On the other hand, for any $\kappa_1\leq\kappa_2$, we also have the estimate
\begin{equation}\label{eq:semigroup_interpolation}
  \big\|S_tx\big\|_{\cH_{\kappa_2}}\lesssim \frac{\|x\|_{\cH_{\kappa_1}}}{t^{\kappa_2-\kappa_1}},
\end{equation}
uniformly in $t>0$. 

\subsection{Mild H\"older Spaces and Mild Sewing Lemma}

Given a function $f:[0,T]\to\cH$, we set $\delta f_{s,t}= f_t-f_s$. We also write
\begin{equation}
\hat\delta f_{s,t}= f_t-S_{t-s}f_s
\end{equation}
for the \emph{mild} increment of $f$.
Then for $s\le t$, it relates to the usual increment process $\delta f_{s,t}=f_t-f_s$ as folllows:
\begin{equation}\label{eq:increment_relation}
  \delta f_{s,t}=\hat \delta f_{s,t}+ (\id-S_{t-s})f_s
\end{equation}
For $\gamma\in(0,1]$  we define the \emph{mild H\"older space} 
\begin{equation}
  \hat\C^\gamma=\left\{f:[0,T]\to\cH:\,\|f\|_{\hat\C^\gamma}\define\sup_{s,t\in[0,T]}\frac{\|\hat\delta f_{s,t}\|}{|t-s|^\gamma}<\infty\right\}.
\end{equation}
The notation $ \hat\C^\gamma\big([0,T],\cH\big)$ is also used if we want to emphasize the underlying space or time interval.  

Note that $\hat\C^\gamma\big([0,T],\cH\big)\hookrightarrow\C\big([0,T],\cH\big)$ embeds continuously. In fact, the continuity of $f$ follows from the strong continuity of the semigroup in combination with \eqref{eq:increment_relation}.

Fix a deterministic function $\h\in\C^\gamma\big([0,T],\cK\big)$ where $\gamma>\f 12$. 
Let $f:[0,T]\times\cH\to\cH$ and $g:[0,T]\times\cH\to L(\cK,\cH)$. We want to show that under suitable regularity conditions on $f$ and $g$, the evolution equation
\begin{equation}\label{eq:young_spde}
  dx_t=\big(Ax_t+f(t,x_t)\big)\,dt+g(t,x_t)\,d\h_t,\qquad t\in[0,T],
\end{equation}
has a unique mild solution. Since any trace-class fBm belongs to $\C^{H-} $ by \cref{lem:uniform_convergence} below, the deterministic results in the sequel carry over to almost sure statements when $\h$ is replaced by $B$.

We understand the integral $\int_0^t S_{t-s}g(s,x_s)\,d\h_s$ appearing in the formulation of the mild solution to \eqref{eq:young_spde} as a Young integral. In particular, we shall make frequent use of the next result which was proven in the classical work of Young  \cite{Young1936}, see also Lyons \cite{Lyons1998,Lyons2002} for a contemporary proof. See also \cite{Zahle}.
\begin{proposition}[Young integral]\label{prop:young}
Let $f\in\C^\alpha\big([0,T],L(\cK,\cH)\big)$ and $\h\in\C^\gamma\big([0,T],\cK\big)$ with $\alpha+\gamma>1$. Then, for each $s,t\in [0,T]$ the integral
\begin{equation*}
  \int_s^t f(r)\,d\h_r\define\lim_{|P|\to 0}\sum_{[u,v]\in P}f(u)\big(\h_v-\h_u\big)
\end{equation*}
exists as a limit in $\cH$ along an arbitrary sequence of partitions of $[s,t]$ with mesh tending to zero. Moreover, we have the estimate
\begin{equation}\label{eq:young}
  \left\|\int_s^t f(r)\,d\h_r-f(s)\big(\h_t-\h_s\big)\right\|_\cH\lesssim\|f\|_{\C^\alpha}\|\h\|_{\C^\gamma}|t-s|^{\alpha+\gamma},
\end{equation}
uniform in $s,t\in[0,T]$.
\end{proposition}

While the estimate \eqref{eq:young} has been successfully used to prove pathwise existence results for (random) Young differential equations in finite dimensions, the following computation shows that it is ill-suited for our investigation of mild solutions: 
\begin{align*}
  \left\|\int_s^t S_{t-r}f(r)\,d\h_r\right\|&\lesssim \|\h\|_{\C^\gamma}\bigg(\sup_{u\in[s,t]}\|S_{t-u}f(u)\|+\sup_{u<v\in[s,t]}\frac{\|S_{t-v}f(v)-S_{t-u}f(u)\|}{|v-u|^\alpha}\bigg)|t-s|^\gamma\\
  &\lesssim \|\h\|_{\C^\gamma}\Big(\|f\|_{\infty}+\|\hat\delta f\|_{\C^\alpha}\Big)|t-s|^\gamma,
\end{align*}
where we used $\sup_{t\le T}  \|S_t\|<\infty$.  Now to show that $\|\hat\delta f\|_{\C^\alpha}<\infty$, we write
\begin{equation*}
  \|\hat\delta f_{s,t}\|_{\cH}\leq\|f(t)-f(s)\|_\cH+\big\|(\id-S_{t-s})f(s)\big\|_{\cH}.
\end{equation*}
While the first term can be handled by imposing appropriate assumptions on $f$, the second term in general only admits an $\mathcal{O}(|t-s|^\alpha)$-bound if we go up the ladder of interpolation spaces; \emph{viz.} $\big\|(\id-S_{t-s})f(s)\big\|_{\cH}\lesssim \|f\|_{\cH_{\alpha}}|t-s|^\alpha$, see \eqref{eq:holder_semigroup}. This is of course notoriously bad for any kind of fixed point argument aimed at showing well-posedness for \eqref{eq:young_spde}. Hence, we resort to a technique introduced by Gubinelli and Tindel in \cite{Gubinelli2010}. We give a presentation of their results which is however more in the style of the monograph of Friz and Hairer \cite{Friz2020}.

Let $\Delta^2=\big\{(s,t)\in[0,T]^2:\,s\leq t\big\}$ and $\Delta^3 =\big\{(r,s,t)\in[0,T]^3:\,r\leq s\leq t\big\}$ be  the two- (respectively  three-) dimensional simplices.
For $\Xi:\Delta^2\to\cH$ and $(r,s,t)\in\Delta^3$ let
\begin{equation}
  \hat\delta\Xi_{r,s,t}=\Xi_{r,t}-\Xi_{s,t}-S_{t-s}\Xi_{r,s}.
\end{equation}
For $\bar\gamma>0$ and $\cH$ a Hilbert space, we set
\begin{equation*}
  \hat{\sC}^{\bar\gamma}\big(\Delta^2,\cH\big)=\left\{\Xi:\Delta^2\to\cH:\,\|\Xi\|_{\hat{\sC}^{\bar\gamma}}=\sup_{0\leq r\leq s\leq t\leq T}\frac{\|\hat\delta\Xi_{r,s,t}\|}{|t-r|^{\bar\gamma}}<\infty\right\}.
\end{equation*}
In the following, we take the Hilbert space to be $\cH_\kappa$, an interpolation space defined by $\A$.  The following version of Gubinelli's sewing lemma \cite{Gubinelli-lemma,Feyel-delaPradelle} can be found in Gubinelli  \cite[Theorem 3.5]{Gubinelli2010}, see also  Gerasimovics and Hairer \cite{Gerasimovics2019}. 
\begin{proposition}[Mild Sewing lemma]\label{prop:mild_sewing}
  Let $\kappa\in\R$ and $0<\alpha\leq 1<\mu$. Let $\Xi:\Delta^2\to\cH_\kappa$ satisfy
  \begin{equation*}
    \sup_{s<t}\frac{\|\Xi_{s,t}\|_{\cH_\kappa}}{|t-s|^\alpha}<\infty
  \end{equation*}
  and suppose that there is a $\Lambda:\Delta^3\to\cH_\kappa$ such that $\hat\delta\Xi_{r,s,t}=S_{t-s}\Lambda_{r,s,t}$ for all $(r,s,t)\in\Delta^3$. Furthermore assume that there is a $C>0$ such that
  \begin{equation*}
    \|\Lambda_{r,s,t}\|_{\cH_\kappa}\leq C |t-r||t-s|^{\mu-1} \qquad\forall\,(r,s,t)\in\Delta^3.
  \end{equation*}
  Then the limit
  \begin{equation*}
    \I\Xi_{s,t}\define\lim_{|P|\to 0}\sum_{[u,v]\in P}S_{t-v}\Xi_{u,v}
  \end{equation*}
  exists along any sequence of partitions $P$ of $[s,t]$ with mesh tending to zero. Moreover, for each $\varpi\in[0,\mu)$, we have the bound
  \begin{equation*}
    \big\|\I\Xi_{s,t}-\Xi_{s,t}\big\|_{\cH_{\kappa+\varpi}}\lesssim C|t-s|^{\mu-\varpi}
  \end{equation*}
  for a hidden prefactor depending only on $\mu$ and $\varpi$.
\end{proposition}

An immediate consequence of \cref{prop:mild_sewing} is the following corollary which allows us to trade time-regularity for interpolation space-regularity:
\begin{corollary}[Mild Young bound]\label{cor:mild_young}
  Let $f\in\hat \C^\alpha\big([0,T],L(\cK,\cH_\kappa)\big)$ and $\h\in\C^\gamma\big([0,T],\cK\big)$ with $\alpha+\gamma>1$. Then, for each $\varpi\in[0,\alpha+\gamma)$, the Young integral of \cref{prop:young} satisfies the bound
  \begin{equation}\label{eq:mild_young}
    \left\|\int_s^t S_{t-r}f(r)\,d\h_r-S_{t-s}f(s)\h_{s,t}\right\|_{\cH_{\kappa+\varpi}}\lesssim \|\h\|_{\C^\gamma}\|f\|_{\hat\C^\alpha}|t-s|^{(\alpha+\gamma-\varpi)\wedge\gamma}.
  \end{equation}
\end{corollary}
\begin{proof}
  The Young integral is given by the limit of the Riemann sum over $\Xi_{u,v}\define S_{v-u}f(u)\h_{u,v}$. Then it is readily checked that $\hat\delta\Xi_{r,u,v}=S_{v-u}\Lambda_{r,u,v}$ for
  \begin{equation*}
    \Lambda_{r,u,v}=\big(S_{u-r}f(r)-f(u)\big)\h_{u,v}= -\hat\delta f_{r,u}\h_{u,v}.
  \end{equation*}
  Moreover, we see that
  \begin{equation*}
    \|\Lambda_{r,u,v}\|_{\cH_{\kappa}}\leq\|f\|_{\hat\C^\alpha} \; \|\h\|_{\C^\gamma} \; |v-r| \; |v-u|^{\alpha+\gamma-1}
  \end{equation*}
  for all $(r,u,v)\in\Delta^3$. Therefore, the required bound follows from \cref{prop:mild_sewing}.
\end{proof}

\subsection{Pathwise Well-Posedness and a Deterministic Stability Lemma}\label{sec:well_posed}

We now turn to a well-posedness result for \eqref{eq:young_spde}. As a preparation we recall the following. Here and in the sequel, we shall denote by $D_xf$ the spatial Fr\'echet derivative of a function $f: [0,T]\times \cH\to \cH$, where $\cH$ is a Hilbert space.
\begin{lemma}\label{lem:holder_composition}
  Let $f:[0,T]\times\cH\to\cH$ be such that  the functions $\{f(t,\cdot): t\in [0,T]\}$ and their first order spatial derivatives uniformly Lipschitz continuous. In addition, suppose that 
  $t\mapsto D_xf(t, x)$ is uniformly $\alpha$-H\"older continuous. Then we have
  \begin{align*}
    &\big\|f(t,x_1)-f(t,x_2)-f(s,x_3)+f(s,x_4)\big\|\\
    \lesssim& \|x_1-x_2-x_3+x_4\|+\|x_1-x_2\|\Big(|t-s|^\alpha+\|x_1-x_3\|+\|x_2-x_4\|\Big).
  \end{align*}
for all $x_1,\dots,x_4\in\cH$ and $s,t\in[0,T]$. 
\end{lemma}
\begin{proof}
  This standard estimate was proved in \cite{Nualart-Rascanu} and in the infinite dimensional setting used in \cite{Maslowski-Nualart}. For the reader's convenience we include the short proof below:

 By assumption there is a constant $L>0$ such that, for all $x,y\in\cH$ and all $t\in[0,T]$,
 \begin{gather*}
    \|f(t,x)-f(t,y)\|\leq L\|x-y\|, \qquad  \|D_x f(t,x)-D_xf(t,y)\|_{L(\cH)}\leq L\|x-y\|, \\
    \|D_x f(s,x)-D_xf(t,x)\|_{L(\cH)}\leq L|t-s|^\alpha.
 \end{gather*}
From the identity 
  \begin{align*}
    &f(t,x_1)-f(t,x_2)-f(s,x_3)+f(s,x_4)\\
    =&\int_0^1 D_x f\big(t,\theta x_1+(1-\theta)x_2\big)(x_1-x_2)\,d\theta-\int_0^1 D_x f\big(s,\theta x_3+(1-\theta)x_4\big)(x_3-x_4)\,d\theta\\
    =&\int_0^1 D_x f\big(s,\theta x_3+(1-\theta)x_4\big)(x_1-x_2-x_3+x_4)\,d\theta \\
    &+ \int_0^1 \Big[D_x f\big(t,\theta x_1+(1-\theta)x_2\big)-D_x f\big(s,\theta x_3+(1-\theta)x_4\big)\Big](x_1-x_2)\,d\theta,
  \end{align*}
   the claimed estimate follows at once.
\end{proof}
The announced existence and uniqueness result reads as follows:
\begin{proposition}\label{prop:well_posedness}
     Let $\alpha, \gamma$ be  two numbers with $\gamma \in \big(\f 12,1\big]$ and  $\alpha+\gamma >1$.
       Let $\h: [0,T]\to \cK$, $f:[0,T]\times\cH\to\cH$, and  $g:[0,T]\times\cH\to L(\cK,\cH)$ satisfy the following conditions:
  \begin{enumerate} 
  \item\label{it:h} $\h\in\C^\gamma\big([0,T],\cK\big)$.
    \item  For any $s,t\in [0,T] $ and $ x,y\in \cH$,    \begin{equation}\label{eq:assumption_f}
      \big\|f(t,x)-f(s,y)\big\|\leq C\big(|t-s|^\alpha+\|x-y\|\big).
    \end{equation}
    \item  The functions $g(t,\bigcdot):\cH\to L(\cK,\cH)$ are Lipschitz continuous, uniformly in $t\in [0,T]$. Moreover,
    \begin{align*}  
    \sup_{s\in[0,T]}  \big\|g(s,x)-g(s,y)\big\|_{L(\cK,\cH_{-\alpha})} &\lesssim \|x-y\|_{\cH_{-\alpha}}, \\
    \sup_{x\in \cH}   \big\|g(t,x)-g(s,x)\big\|_{L(\cK,\cH_{-\alpha})}&\lesssim |t-s|^\alpha, \\
     \sup_{t\in [0,T]}  \|D_x g(t,x)-D_xg(t,y)\|_{L(\cK,L(\cH_{-\alpha}))}&\lesssim \|x-y\|_{\cH_{-\alpha}}, \\
    \sup_{x\in \cH}  \|D_x g(s,x)-D_xg(t,x)\|_{L(\cK,L(\cH_{-\alpha}))}&\lesssim |t-s|^\alpha,
    \end{align*}
    where the Fr\'echet derivative is taken with respect to the norm of $\cH_{-\alpha}$.
    \end{enumerate}
    Then, for any $x_0\in\cH$, the Young equation \eqref{eq:young_spde} has a unique global mild solution in $\hat\C^{\gamma-}\big([0,T],\cH\big)$. Furthermore, for each $\bar\gamma<\gamma$, there are constants $C>0$ and $N\in\N$ independent of $x_0$ and $\h$ such that
     \begin{equation}\label{eq:deterministic_apriori_bound}
    \|x\|_{\hat\C^{\bar{\gamma}}([0,T],\cH)}\leq C\big(1+\|x_0\|\big)\Big(1+\|\h\|_{\C^\gamma([0,T],\cH)}\Big)^N.
  \end{equation}
  \end{proposition}

  With this preliminary pathwise result, we obtain in \cref{non-uniform-estimate,uniform-Lp-bound} the necessary estimates on the mild solutions of the SPDE (\ref{eq:intro_spde}). Before giving the proof, we shall compare \cref{prop:well_posedness} to previous well-posedness results and give an example satisfying its conditions. In \cite{Maslowski-Nualart} an existence and uniqueness theorem under certain `mild' Lipschitz conditions on both the diffusion coefficient and its derivative was shown for a time-independent  SPDE, in \cite{Hesse-Neamtu}  well-posedness was shown for the rough case with the help of a sewing lemma technique whose condition on $g$ seems to be complementary to ours. Both results are for the time-independent case, here we allow time dependence of H\"older regularity to accommodate that the coefficients in our SDE depend on another stochastic process. The result of Maslowski and Nualart was later improved by Garrido-Atienza, Lu, and Schmalfu{\ss} \cite{Garrido2010}. There the usual Lipschitz conditions on $g$ and its derivative in the norm of $\cH$ are assumed. However, we would like to emphasize that the works \cite{Maslowski-Nualart,Garrido2010} require the square root of the covariance operator of the fBm to be trace-class and only show that the solution exists in H\"older spaces of order strictly less than $\frac12$. Our technique based on \cref{cor:mild_young} allows us to (a) establish the same regularity as in the finite-dimensional case and (b) to relax the unnatural trace condition, c.f. \cref{trace} for further discussion.

  Examples of functions $f$ and $g$ satisfying the conditions of the \cref{prop:well_posedness} can be constructed with of the help of the Nemytskii operator.   Given any  function $G: U\subset \R^n\to \R^n$,
   the Nemytskii operator $T_G :\cH\to\cH$, where $\cH$ is a space of functions on $U$, is
$$(T_G)(x)(\eta)=G\big(x(\eta)\big), \qquad  \eta\in \R^n.$$    
We shall assume $G: \R^n\to \R^n$ is of class $\C_b^\infty$, the space of bounded smooth functions with bounded derivatives of all orders.  
It is sufficient to construct a time-independent example, for time dependency can be easily introduced by, for example, multiplying $f,g$ by $t^\alpha$. In the examples we work with the Lapliacian for concreteness; other elliptic second order differential operator would work \emph{mutatis mutandis}.

\begin{example}\label{ex:r_n}
 Let  $A=\Delta -1$ where $\Delta$ is  the Laplacian on $\R^n$.    Let $H^s$ denote the completion of $C_c^\infty(\R^n)$ under the metric induced by the inner product
 \begin{equation}\label{eq:sobolev_inner_product}
   \<x,y\>_{H^s}=  \sum_{0\le i\le s}\int_{\R^n} \Braket{D^i x(\eta), D^i y(\eta)}\,d\eta
 \end{equation}
 where $s\in\N_0$. In the sequel let $n=1$ for concrete computations. We take $\cH=H^2$ and $\alpha=\frac12$, so $\cH_{-\alpha}=H^1$.
 In general, we can take $G \in \C_b^\infty(\R)$ with $G(0)=0$; for example let $G(t)=\sin(t)$. Then $T_G \in L^2$, 
 $\|T_G(x)-T_G(y)\|_{H^1}\lesssim |G^\prime|_\infty\|x-y\|_{H^1}+|G^{\prime\prime}|_\infty\|y\|_{H^1} \sup_t  |x(t)-y(t)| $.  Since $  |x-y|_\infty\lesssim \|x-y\|_{H^1}$, $T_G$ is a locally Lipschitz continuous function in $H^1$. Similarly $T_G$ is locally Lipschitz in $H^2$  which follows from 
 \begin{align}
  \big\|T_G(x)-T_G(y)\big\|_{H^2}&\lesssim \big\|T_G(x)-T_G(y)\big\|_{H^1}
  +|G''|_\infty\|x-y\|_{H^2} \big(|\dot y|_\infty+|\dot x|_\infty\big)\nonumber\\
  &\phantom{\lesssim}+\big|G^{(3)}\big|_\infty|x-y|_{\infty}\|y\|_{H^1}^2+|G'|_\infty\|x-y\|_{H^2}
  +|G''|_\infty|x-y|_\infty\|y\|_{H^2}\nonumber\\
  &\lesssim |G|_{\C_b^3} \big(1+\|x\|_{H^2}^2+\|y\|_{H^2}^2\big) \;  \|x-y\|_{H^2}.\label{eq:tg_estimate}
 \end{align}
 Since the Fr\'echet derivative of $T_G: H^1\to H^1$ is given by pointwise differentiation, that is, $\big((DT_G)_x(v)\big)(t)=( DG)_{x(t)}\big(v(t)\big)$ where $x,v\in H^1$, the
local Lipschitz continuity of $D T_G$  follows similarly. We then take a smooth function  $\phi: \R\to \R_+$ with compact support  and  a driving process $\h$ from an arbitrary separable Hilbert space \cK.
  Let  $(h_i)_{i\in\N}$ denote an orthonormal basis of $\cK$. For $h\in \cK$ set 
   \begin{equation}\label{example-construction}
     f(x)=\phi\big(\|x\|^2_\cH\big) T_{\psi_0}(x), \qquad    g(x)(h)=\sum_{i\in \N}a_i  \<h, h_i\>\phi\big(\|x\|^2_{\cH_{-\alpha}}\big) T_{\psi_i}(x),
   \end{equation} 
   where $\{\psi_i, i\in\N_0\}$ is a sequence from $ \C_b^\infty(\R)$ vanishing at $0$ and with all bounds uniform in $i$ and $\sum_{i\in\N} |a_i|^2<\infty$. By the estimate \eqref{eq:tg_estimate}, $f$ is globally Lipschitz on $\cH=H^2$ and
   \begin{align*}\|g(x)-g(y)\|_{L(\cK,\cH_{-\alpha})}&=\sup_{\substack{h\in\cK\\ \|h\|_\cK=1}}\left\|\sum_{i\in \N}a_i  \<h, h_i\> 
   \Bigl ( \phi\big(\|x\|^2_{\cH_{-\alpha}}\big) T_{\psi_i} (x)-\phi\big(\|y\|^2_{\cH_{-\alpha}}\big) T_{\psi_i }(y)\Bigr)\right\|_{\cH_{-\alpha}}\\
   &\lesssim\sup_{i\in\N}|\psi_i|_{\C_b^3}\|x-y\|_{\cH_{-\alpha}}\sum_{i\in\N} \big|a_i\<h, h_i\>\big| \lesssim\|x-y\|_{\cH_{-\alpha}},
   \end{align*}
   where the last step follows since the series is finite by the Cauchy-Schwarz inequality.   Similarly, one checks that $D g$ is also globally Lipschitz continuous on $L\big(\cK,L(\cH_{-\alpha})\big)$. The requirements of \cref{prop:well_posedness} are therefore satisfied.
\end{example}

For $n>1$ we can argue similar by employing higher order Sobolev spaces and abstract embedding theorems. A good source for this is  \cite{Runst1996}:
\begin{example} Let $\T^n$ be the $n$-dimensional torus and $A=\triangle$ be the Laplacian on the torus. Let $H^s(\T^n)$, $s\in\N$, be the Sobolev spaces on the torus, that is, the completion of $\C_c^\infty(\T^n)$ under the scalar product \eqref{eq:sobolev_inner_product} with integration over $\T^n$. Fix $k>\frac{n}{2}+2$. We take $\cH\define H^{k}(\T^n)$. The interpolation spaces are given by $\cH_\kappa=H^{k+2\kappa}(\T^n)$, where for non-integer order $\kappa$
it is defined by the norm $(\sum_{m\in \mathbb{Z}^n} |m|^{2s}|\hat x(m)|^2)^{\f 12}$ where $\hat x(m)=\int x(\eta)e^{-2\pi im\cdot \eta}\,d\eta$ is the Fourier coefficient. We work with $\alpha=\frac12$, note  $\cH_{-\f 12}=H^{k-1}(\T^n)$. 
  Let $G\in\C^\infty(\T^n)$.  Then for any natural number $\ell$ and any $v=(v_1,\dots, v_\ell)\in (\R^n)^\ell$,  the weak derivatives are given by:
  $$D^{\ell}\big(T_G(x)\big)_{\eta}(v)=(DG)_{x(\eta)}\big((D^\ell x)_\eta (v)\big)+\sum_{j=2}^{\ell-1} \Big(D^{\ell-j+1}G\Big)_{x(\eta)} \mathfrak{P}_j(v),
  $$
  where $\mathfrak{P}_j(v) $ involves products of derivatives of $x$ up to total order $j$. Note that $DT_G(x)$ is trivially in $L^2$. For higher order derivatives, we note that any $x\in H^s$ with $s>\f n 2+1$ is in $W^{1,q}$ for any $q\ge 2$ 
This follows from the Sobolev inequalities on the torus:   $\|x\|_{L^q}\lesssim \|x\|_{H^{\f n2 -\f n q}}$  and its application to $Dx$ yields  $\|Dx\|_{L^q} \lesssim \|x\|_{H^{\f n 2 -\f n q+1}}$. A nice proof for the Sobolev inequality can be found in \cite{Benyi-Oh}.
By \cite[Theorem in \para 5.2.5]{Runst1996}, see also \cite[Proposition 3.1]{Hofmanova2013},
$$\|T_G(x)\|_{H^\ell} \lesssim  (1+\|x\|_{W^{1, 2\ell}}^\ell+\|x\|_{H^\ell}).$$
  It is now evident that $T_g(H_k)\subset H_k$ for any $k>\f n 2+1$ and in particular $T_G(\cH)\subset \cH$ and $T_G(\cH_{-\alpha})\subset\cH_{-\alpha}$.
Finally using the inequality 
$|x|_\infty \lesssim \|x\|_{H^k}$ for $k>\f n 2$ and a similar argument to that in \cref{ex:r_n},  we see that $x\mapsto T_G(x)$ is locally Lipschitz continuous on such $H^k$ and therefore on both $\cH$ and $\cH_{-\alpha}$. So is $x\mapsto DT_G(x)$.
  Similar to \cref{ex:r_n}, $T_G$ is infinitely often Fr\'echet differentiable, both as a mapping on $\cH$ and $\cH_{-\alpha}$, and $D^\ell T_G(x)= G^{(\ell)}\circ x$ for each $\ell\geq 1$.

 Finally we construct a globally Lipschitz mapping $f$ and transfer this construction to the diffusion vector field $g$. Take again a smooth function $\phi: \R\to \R_+$ with compact support.  Fix a separable Hilbert space $\cK$ supporting the driver $\h$ of the equation. Let $(h_i)_{i\in\N}$ be an orthonormal basis of $\cK$ and set 
 for $x\in\cH$ and $h\in \cK$:
    \begin{equation}\label{eq:def_example_g}
   f(x)\define\phi\big(\|x\|_{\cH}^2\big)T_F(x),\qquad
  g(x)(h)\define \sum_{i\in \N} a_i \<h,h_i\>\phi\big(\|x\|_{\cH_{-\alpha}}^2\big)T_{G_i}(x),
  \end{equation}
  where $(a_i)_{i\in\N}$ is a square summable sequence and $G_i:\T^n\to\R$ are smooth. Since $\cH\hookrightarrow\cH_{-\alpha}$, $f$ is $\C_b^\infty$ in the Fr\'echet sense both on $\cH$ and $\cH_{-\alpha}$. We also see that $g(x):\cK\to\cH$ defines a bounded linear operator. Moreover, the map $g$ is smooth in the Fr\'echet sense, both on $\cH$ and $\cH_{-\alpha}$.   If $\sup_{i\in\N}\sup_{|y|\leq R}\big|G_i^{(\ell)}(y)\big|<\infty$ for each $\ell\geq 0$, where we defined $R\define\sup\left\{|x|_\infty:\, x\in\cH_{-\alpha}\,\&\,\|x\|_{\cH_{-\alpha}}^2\in\supp(\phi)\right\}$, the global Lipschitz continuity in the respective norms required by \cref{prop:well_posedness} follows.
\end{example}

\begin{proof}[Proof of \cref{prop:well_posedness}]
  Fix $\bar\gamma<\gamma$. Without loss of generality, we may assume that $\alpha<\gamma$. 
  \begin{enumerate}
    \item\label{it:a_priori} We first establish the \emph{a priori} bound \eqref{eq:deterministic_apriori_bound} on $\|x\|_{\hat\C^{\bar\gamma}}$, showing that the solution (if it exists) cannot blow up in finite time.
    \item\label{it:fixed_point} We show that the application
    \begin{equation*}
     ( \A y)_t\define S_t y_0+\int_0^t S_{t-s}f(s,y_s)\,ds+\int_0^t S_{t-s} g(s,y_s)\,d\h_s,\qquad t\in[0,T],
    \end{equation*}
    is contracting on the complete metric space
   \begin{equation*}
      V\define\left\{y:[0,\rho]\to\cH:\,y_0=x_0,\,\|y\|_{\hat\C^{\bar\gamma}([0,\rho],\cH)}\leq 1\right\},
    \end{equation*}
        provided we choose $\rho\in(0,T]$ small enough. Hence, the Banach fixed point theorem implies that the equation \eqref{eq:young_spde} has a unique mild solution on the interval $[0,\rho]$. Owing to \ref{it:a_priori} we can iterate this procedure to construct the unique mild solution on all of $[0,T]$.
  \end{enumerate}

  \dashuline{\textbf{Step \ref{it:a_priori}:}} Let $\rho\in(0,T]$ and $x_0\in\cH$. From 
  $$x_t=S_tx_0+\int_0^tS_{t-s}f(s,x_s)\,ds+\int_0^t S_{t-s}g(s,x_s)\,d\h_s=S_tx_0+\hat{\delta}x_{0,t}$$
  and by the uniform boundedness of the semigroup,  we have
  \begin{equation}\label{eq:first_step_a_priori}
    \|x_t\|\leq \bigl\|S_tx_0 \bigr\| +\bigl\|\hat\delta x_{0,t}\bigr\|
    \lesssim\|x_0\|+ t^\gamma \|x\|_{\hat\C^{\bar{\gamma}}}\lesssim \|x_0\|+\|x\|_{\hat\C^{\bar{\gamma}}},\qquad \forall\,t\in[0,\rho].
  \end{equation}
     To obtain the required a priori bound we begin with $x=\A x$ and observe that
  \begin{equation*}
     \hat\delta(\mathcal{A}x)_{s,t}=\int_s^t S_{t-s}f(r,x_r)\,dr+\int_s^t S_{t-r}g(r,x_r)\,d\h_r,
  \end{equation*}
  Since $\|f(r,x)\|_\cH\lesssim 1+\|x\|_\cH$ by the Lipschitz continuity of $f$, we have the trivial bound:
  \begin{equation}\label{eq:a_priori_integral}
    \left\|\int_s^t S_{t-r}f(r,x_r)\,dr\right\|_{\cH}\lesssim\big(1+\|x_\cdot \|_{\infty}\big) |t-s|.
  \end{equation}
  Next, we use the Young bound from \cref{cor:mild_young} with $\varpi=\alpha$ to find
  \begin{equation}\label{eq:a_priori_young}
    \left\|\int_s^t S_{t-r}g(r,x_r)\,d\h_r\right\|_{\cH}\lesssim\|\h\|_{\C^\gamma}\left(\big\|g(\bigcdot,x_{\bigcdot})\big\|_{\infty}+\big\|g(\bigcdot,x_{\bigcdot})\big\|_{\hat\C^\alpha(L(\cK,\cH_{-\alpha}))}\right)|t-s|^\gamma.
  \end{equation}
  Note that $\big\|g(\bigcdot,x_{\bigcdot})\big\|_{\infty}\lesssim 1+\|x\|_\infty$ by uniform Lipschitz continuity. We split the second term in the bracket on the right hand side of \eqref{eq:a_priori_young}:
  \begin{equation*}
    \big\|\hat\delta g(\bigcdot,x_{\bigcdot})_{s,t}\big\|_{L(\cK,\cH_{-\alpha})}\leq\big\|g(t,x_t)-g(s,S_{t-s}x_s)\big\|_{L(\cK,\cH_{-\alpha})}+\big\|g(s,S_{t-s}x_s)-S_{t-s}g(s,x_s)\big\|_{L(\cK,\cH_{-\alpha})}.
  \end{equation*}
  Firstly, by the uniform H\"older continuity of $g$ in time and its Lipschitz continuity  on $\cH_{-\alpha}$,
  \begin{equation*}
    \big\|g(t,x_t)-g(s,S_{t-s}x_s)\big\|_{L(\cK,\cH_{-\alpha})}\lesssim\Big(|t-s|^\alpha+\big\|\hat\delta x_{s,t}\big\|_{\cH_{-\alpha}}\Big)\lesssim\Big(1+\|x\big\|_{\hat\C^\gamma([0,\rho],\cH)}\Big)|t-s|^\alpha,
  \end{equation*}
  where in the last step we used that $\cH_{-\alpha}\hookrightarrow\cH$ and that $\alpha<\gamma$.  Moreover by \eqref{eq:holder_semigroup} and Lipschitz continuity of $g(s,\bigcdot):\cH_{-\alpha}\to L(\cK,\cH_{-\alpha})$, we have
  \begin{align*}
    &\big\|g(s,S_{t-s}x_s)-S_{t-s}g(s,x_s)\big\|_{L(\cK,\cH_{-\alpha})}\\
    &\leq \big\|g(s,S_{t-s}x_s)-g(s,x_s)\big\|_{L(\cK,\cH_{-\alpha})}+\big\|g(s,x_s)-S_{t-s}g(s,x_s)\big\|_{L(\cK,\cH_{-\alpha})}\\
    &\lesssim\big\|(\id-S_{t-s})x_s\big\|_{\cH_{-\alpha}}+\big\|(\id-S_{t-s})g(s,x_s)\big\|_{L(\cK,\cH_{-\alpha})}\\
    &\lesssim\Big(\|x \|_{\infty}+\big\|g(\bigcdot,x_{\bigcdot})\big\|_{\infty}\Big)|t-s|^\alpha\lesssim\big(1+\|x\|_\infty\big)|t-s|^\alpha.
  \end{align*}
  In summary, since  $\|x\|_\infty
    \lesssim\|x_0\|+  \|x\|_{\hat\C^{\bar{\gamma}}}$ by \eqref{eq:first_step_a_priori},
  \begin{equation*}
  \big\|g(\bigcdot,x_{\bigcdot})\big\|_{\hat\C^\alpha(L(\cK,\cH_{-\alpha}))}
  \lesssim \big(1+\|x_0\|+\|x\|_{\hat\C^{\bar{\gamma}}}\big).
  \end{equation*}
  Inserting this back into \eqref{eq:a_priori_integral}--\eqref{eq:a_priori_young}, we have shown that
    \begin{equation*}
    \left\|\int_s^t S_{t-r}g(r,x_r)\,d\h_r\right\|_{\cH}\lesssim\|\h\|_{\C^\gamma}\left(1+\|x_0\|+\|x\|_{\hat\C^{\bar{\gamma}}}\right)|t-s|^\gamma.
  \end{equation*}
and
  \begin{equation}\label{eq:invariance}
    \|x\|_{\hat\C^{\bar{\gamma}}([0, \rho], H)}= \big\|\A x \big\|_{\hat\C^{\bar{\gamma}}([0,\rho],\cH)} \lesssim\Big(1+\|\h\|_{\C^\gamma}\Big)\Big(1+\|x_0\|+\| x\|_{\hat\C^{\bar\gamma}([0, \rho], H)}\Big)\rho^{\gamma-\bar\gamma}, 
  \end{equation}
  whence, upon choosing $\rho>0$ small enough,
  \begin{equation*}
    \|x\|_{\hat\C^{\bar{\gamma}}([0,\rho],\cH)}\leq C\big(1+\|x_0\|\big)\Big(1+\|\h\|_{\C^\gamma([0,T],\cH)}\Big).
  \end{equation*}
   This bound can be iterated to obtain
  \begin{equation*}
    \|x\|_{\hat\C^{\bar{\gamma}}([t,t+\rho],\cH)}\leq C^{\left[\frac{t}{\rho}\right]+1}\big(1+\|x_0\|\big)\Big(1+\|\h\|_{\C^\gamma([0,T],\cH)}\Big)^{\left[\frac{t}{\rho}\right]+1},
  \end{equation*}
   where $[a]$ denotes the integer part of a number $a$.  
  
  Now  let $\gamma\in(0,1)$ and $\rho\in(0,1]$, we show that
  \begin{equation}
    \|x\|_{\hat\C^\gamma([0,T],\cH)}\lesssim \frac{1}{\rho^{1-\gamma}}\sup_{t\in[0,T-\rho]}\|x\|_{\hat\C^\gamma([\rho,t+\rho],\cH)}.
  \end{equation}
  Fix $0\leq s<t\leq T$. Without loss of generality we can assume that $|t-s|>\rho$. Let $P$ be a partition of $[s,t]$ with $|t-s|\rho^{-1}+1$ elements and mesh $|P|\leq\rho$. For concreteness choose for instance $\{(s+n\rho)\wedge t:n=0,\dots,|t-s|\rho^{-1}\}$. Then
  \begin{align*}
    \|\hat\delta x_{s,t}\|_{\cH}&\leq\sum_{[u,v]\in P}\big\|S_{t-v}\hat\delta x_{u,v}\big\|_{\cH}\lesssim\sum_{[u,v]\in P}\big\|\hat\delta x_{u,v}\big\|_{\cH}\\
    &\leq\rho^{\gamma}\left(\frac{|t-s|}{\rho}+1\right)\sup_{t\in[0,T-\rho]}\|x\|_{\hat\C^\gamma([\rho,t+\rho],\cH)}\\
    &\leq 2\rho^{\gamma}|t-s|\sup_{t\in[0,T-\rho]}\|x\|_{\hat\C^\gamma([\rho,t+\rho],\cH)}.
  \end{align*}
 This finally implies the apriori estimate 
 \begin{equation*}
    \|x\|_{\hat\C^{\bar\gamma}([0,T],\cH)}\lesssim \frac{C^{\left[\frac{T}{\rho}\right]+1}}{\rho^{1-\bar\gamma}}\big(1+\|x_0\|_\cH\big)\Big(1+\|\h\|_{\C^\gamma([0,T],\cH)}\Big)^{\left[\frac{T}{\rho}\right]+1}.
  \end{equation*}

  \vspace{1em}
  \dashuline{\textbf{Step \ref{it:fixed_point}:}} The invariance of the set $V$ under the application $\mathcal{A}$, $\mathcal{A}(V)\subset V$, has already been shown in step \ref{it:a_priori}. In fact, it follows immediately from \eqref{eq:invariance} that $\|\mathcal{A}y\|_{\hat\C^{\bar\gamma}}\leq 1$ for $\rho>0$ sufficiently small. We are thus left to show that there is a $C<1$ and $\rho>0$ such that
  \begin{equation}\label{eq:contraction}
    \big\|\mathcal{A} y-\mathcal{A}\bar y\big\|_{\hat\C^{\bar\gamma}([0,\rho],\cH)}\leq C\|y-\bar y\|_{\hat\C^{\bar\gamma}([0,\rho],\cH)}, \qquad \qquad \forall y,\bar y\in V.
  \end{equation}
  By the Lipschitz continuity of $f$,
\begin{equation}
\left\|\int_s^t S_{t-r}\big(f(r,y_r)-f(r,\bar y_r)\big)\,dr\right\|\lesssim \|y-\bar y\|_{\infty}|t-s|. \label{eq:contraction_1}
\end{equation}    
Similarly to \eqref{eq:a_priori_young}, we also find  
      \begin{align}
    &\left\|\int_s^t S_{t-r}\big(g(r,y_r)-g(r,\bar y_r)\big)\,d\h_r\right\|\nonumber\\
    \lesssim& \|\h\|_{\C^\gamma}\left(\big\|g(\bigcdot,y_{\bigcdot})-g(\bigcdot,\bar y_{\bigcdot})\big\|_{\infty}+\big\|g(\bigcdot,y_{\bigcdot})-g(\bigcdot,\bar y_{\bigcdot})\big\|_{\hat\C^\alpha(L(\cK,\cH_{-\alpha}))}\right)|t-s|^\gamma.\label{eq:contraction_2}
  \end{align}
  By assumption, $\big\|g(\bigcdot,y_{\bigcdot})-g(\bigcdot,\bar y_{\bigcdot})\big\|_{\infty}\lesssim\|y -\bar y\|_{\infty}$. On the other hand, we can write
  \begin{align*}
    \hat\delta\big(g(\bigcdot,y_{\bigcdot})-g(\bigcdot,\bar y_{\bigcdot})\big)_{s,t}&=\Big[g(t,y_t)-g(t,\bar y_t)-g(s,S_{t-s}y_s)+g(s,S_{t-s}\bar{y}_s)\Big] \\
    &\phantom{=}+ \Big[g(s,S_{t-s}y_s)-g(s,S_{t-s}\bar y_s)-g(s,y_s)+g(s,\bar y_s)\Big]\\
    &\phantom{=}+ \Big[g(s,y_s)-g(s,\bar y_s)-S_{t-s}g(s,y_s)+S_{t-s}g(s,\bar y_s)\Big]\\
    &=\rom{1}_{s,t}+\rom{2}_{s,t}+\rom{3}_{s,t}.
  \end{align*}
  Next, we apply \cref{lem:holder_composition} to the first two terms. 
  Since $y, \bar y\in V$, this gives
  \begin{align*}
    \|\rom{1}_{s,t}\|_{L(\cK,\cH_{-\alpha})}\lesssim \big\|\hat\delta (y-\bar y)_{s,t}\big\|_{\cH_{-\alpha}}+\|y_t-\bar{y}_t\|_{\cH_{-\alpha}}\left(|t-s|^\alpha+\big\|\hat\delta y_{s,t}\big\|_{\cH_{-\alpha}}+\big\|\hat\delta \bar y_{s,t}\big\|_{\cH_{-\alpha}}\right)\\
    \lesssim \Big(\|y-\bar y\|_{\hat\C^{\bar\gamma}}+\|y_t-\bar{y}_t\|_{\cH_{-\alpha}}\Big)|t-s|^\alpha,
  \end{align*}
   so that $\|\rom{1}_{s,t}\|_{\C^\alpha(L(\cK,\cH_{-\alpha}))}\lesssim \big(1+\|x_0\|\big)\|y-\bar y\|_{\hat\C^{\bar\gamma}(\cH)}|t-s|^\alpha$.  Similarly, we find
  \begin{align*}
    \|\rom{2}_{s,t}\|_{L(\cK,\cH_{-\alpha})}&\lesssim \big\|(\id-S_{t-s})(y_s-\bar y_s)\big\|_{\cH_{-\alpha}}\\
    &\phantom{\lesssim}+\big\|S_{t-s}(y_s-\bar y_s)\big\|_{\cH_{-\alpha}}\left(\big\|(\id-S_{t-s})y_s\big\|_{\cH_{-\alpha}}+\big\|(\id-S_{t-s})\bar y_s\big\|_{\cH_{-\alpha}}\right)\\
    &\lesssim  \|y-\bar y\|_{\infty} \Big(\|y\|_{\infty}+\|\bar y\|_{\infty}\Big)|t-s|^\alpha\\
    &\lesssim\big(1+\|x_0\|\big)\|y-\bar y\|_{\hat\C^{\bar{\gamma}}(\cH)}|t-s|^\alpha,
  \end{align*}
  where we made use of \eqref{eq:holder_semigroup}. Finally, we observe that
  \begin{align*}
    \|\rom{3}_{s,t}\|_{L(\cK,\cH_{-\alpha})} &=  \big\|(\id-S_{t-s}) \big(g(s,y_s)-g(s,\bar{y}_s)\big)\big\|_{L(\cK,\cH_{-\alpha})} \\
   & \lesssim \big\|g(s,y_s)-g(s,\bar{y}_s)\big\|_{L(\cK, \cH)}|t-s|^\alpha
    \lesssim \big(1+\|x_0\|\big)\|y-\bar y\|_{\hat\C^{\bar\gamma}}|t-s|^\alpha.
  \end{align*}
  Combining these bounds, we see that
  \begin{equation*}
    \big\|g(\bigcdot,y_{\bigcdot})-g(\bigcdot,\bar y_{\bigcdot})\big\|_{\hat\C^\alpha([0,\rho], L(\cK,\cH_{-\alpha}))}\lesssim \big(1+\|x_0\|\big)\|y-\bar y\|_{\hat\C^{\bar\gamma}},
  \end{equation*}
  whence we conclude
  \begin{equation}\label{eq:stability_difference}
    \big\|\mathcal{A} y-\mathcal{A}\bar y\big\|_{\hat\C^{\bar\gamma}}\lesssim \rho^{\gamma-\bar\gamma}\|y-\bar y\|_{\hat\C^{\bar\gamma}}
  \end{equation}
  from \eqref{eq:contraction_1}--\eqref{eq:contraction_2}. This establishes \eqref{eq:contraction} upon potentially further decreasing $\rho>0$ and the proof is complete.
\end{proof}

\begin{remark}
  From the proof of  \cref{prop:well_posedness} we note if  the global Lipschitz continuity of $f,g$ is replaced by the local Lipschitz continuity,  the equation has a local solution. This is in fact sufficient for our purpose below, leading to a local instead of a global averaging theorem.
\end{remark}

Finally, the following Gr\"onwall-type stability result will be used in the proof of the averaging principle in \cref{sec:slow_fast}:
\begin{lemma}[Mild residue lemma]\label{lem:residue}
  Let $\gamma\in\big(\frac12,1\big]$, $\alpha\in(1-\gamma,\gamma)$. Suppose that $f$, $g$, and $\h$ satisfy the conditions of \cref{prop:well_posedness}. Then, for any $y,\bar y\in\hat\C^\alpha\big([0,T],\cH\big)$, the equations
  \begin{align*}
    x_t&=y_t+\int_0^t S_{t-s} f(s,x_s)\,ds+\int_0^t S_{t-s}g(s,x_s)\,d\h_s,\\
    \bar x_t&=\bar y_t+\int_0^t S_{t-s} f(s,\bar x_s)\,ds+\int_0^t S_{t-s}g(s,\bar x_s)\,d\h_s
  \end{align*}
  have a unique solution in $\hat\C^{\alpha}\big([0,T],\cH\big)$. If in addition $y_0=\bar y_0$, we have the stability estimate
  \begin{equation}\label{eq:residue_stability}
    \|x-\bar x\|_{\hat\C^\alpha\big([0,T],\cH\big)}\leq C\|y-\bar y\|_{\hat\C^\alpha\big([0,T],\cH\big)}
  \end{equation}
where the constant $C$ depends only on $f$, $g$, $\|\h\|_{\C^\gamma}$, $\|y_0\|$, $\|y\|_{\hat\C^\alpha}$, $\|\bar y\|_{\hat\C^\alpha}$, and $T$.
\end{lemma}
\begin{proof}
The well-posedness of the equations for $x$ and $\bar x$ follows as in \cref{prop:well_posedness}. 
On a small interval $[0, \rho]$ this is the consequence of the contraction property of $\A$. By the a priori bound \eqref{eq:deterministic_apriori_bound} of solutions,  the required estimate holds.
\end{proof}

\begin{remark}
  A more refined analysis shows that in \eqref{eq:residue_stability} one could choose
  \begin{equation*}
    C=D\exp\left(D\|\h\|_{\C^\gamma}^{\frac1\gamma}+D\|y\|_{\hat\C^\alpha}^{\frac1\alpha}+D\|\bar y\|_{\hat\C^\alpha}^{\frac1\alpha}\right)
  \end{equation*}
  for a constant $D>0$ depending only on $f$, $g$, $\|y_0\|$, and the terminal time $T>0$; see \cite[Lemma 2.2]{Hairer2020} for details. To prove convergence in probability, it is however enough to have $C<\infty$ with probability $1$, see \cref{sec:slow_fast} below.
\end{remark}

\section{Mild Stochastic Sewing Lemma}
\label{mild-sewing}

Our main endeavor in the next section is to make sense of the integral
\begin{equation*}
  \int_0^{\bigcdot} g(t, X_t)\,dB_t
\end{equation*}
and to derive strong $L^p$-estimates on its H\"older norm. Here and in the sequel, the space $L^p$ is always understood over the randomness. This is based on a variant of L\^e's stochastic sewing lemma \cite{Le2020} adapted to the mild calculus used in this article, which we present in the sequel.

Let $\big(\Omega,\F,(\F_t)_{t\geq 0},\prob\big)$ be a filtered probability space.
Fix a terminal time $T>0$ and let $\mathcal{S}^p(\cH)$,  where $\cH$ is a Hilbert space,  denote the set of adapted, continuous two-parameter stochastic processes on the simplex with finite $p^\text{th}$ moments; in symbols: 
\begin{equation*}
\mathcal{S}^p(\cH)\define\left\{\Sigma:\Delta^2\times\Omega\to\cH: \, (s,t)\mapsto \Sigma_{s,t} \in  L^p(\Omega, \F_t;\cH)\text{ is continuous}\right\}.
\end{equation*}
Given $\eta,\bar{\eta}>0$, we define the spaces
\begin{align*}
    H_\eta^p(\cH)&\define\left\{\Xi\in\mathcal{S}^p(\cH):\,\|\Xi\|_{H_\eta^p}\define\sup_{0\leq s<t\leq T}\frac{\|\Xi_{st}\|_{L^p}}{|s-t|^\eta}<\infty\right\}, \\
    \bar{H}_{\bar{\eta}}^p(\cH)&\define\left\{\Xi\in\mathcal{S}^p(\cH):\,\vertiii{\Xi}_{\bar{H}_{\bar{\eta}}^p}\define\sup_{0\leq s<u<t\leq T}\frac{\|\expec{\hat\delta \Xi_{sut}|\mathcal{F}_s}\|_{L^p}}{|s-t|^{\bar{\eta}}}<\infty\right\},
\end{align*}
where we recall that 
$$\hat\delta \Xi_{s,u,t}=\Xi_{s,t}-\Xi_{u,t}-S_{t-u}\Xi_{s,u}.
$$
In the sequel we take the Hilbert space to be $\cH_\kappa$ and occasionally abbreviate $L^p(\cH_\kappa)$ to $L^p$.
As in L\^e's work, our proof of  the mild stochastic sewing lemma (\cref{prop:stochastic_mild_sewing} below) is based on the following inequality: Let $(Z_i)_{i\in\N}$ be an 
$\cH_\kappa$-valued, discrete-time stochastic process adapted to a filtration $(\F_i)_{i\in\N_0}$ and let $d_i=Z_i-\Expec{Z_i\,|\,\F_{i-1}}$ be the sequence of martingale differences. Then
\begin{align}
  \left\|\sum_{i=1}^n Z_i\right\|_{L^p}&\leq\left\|\sum_{i=1}^n \Expec{Z_i\,|\,\F_{i-1}}\right\|_{L^p}+\left\|\sum_{i=1}^n d_i\right\|_{L^p} \nonumber\\
  &\lesssim\sum_{i=1}^n\big\|\Expec{Z_i\,|\,\F_{i-1}}\big\|_{L^p}+\Big\|\sum_{i=1}^nd_i^2 \Big\|_{L^{p/2}}^{\frac12} \nonumber\\
  &\lesssim\sum_{i=1}^n\big\|\Expec{Z_i\,|\,\F_{i-1}}\big\|_{L^p}+\left(\sum_{i=1}^n\big\|Z_i\big\|_{L^p}^2\right)^{\frac12}. \label{eq:bdg_consequence_2}
\end{align}
In the second line we used Minkowski's and Burkholder's inequalities. The last step follows from another application of Minkowski's inequality. Burkholder's inequality on Hilbert spaces can be proven by the classical square function approach of Burkholder \cite{Burkholder1984}. In particular, the prefactors, which may depend on $p\geq 1$, are \emph{independent} of $n\in\N$.

Before stating the mild version of the sewing lemma,  we first make some basic error estimates.
\begin{lemma}\label{lem:sewing_help}
  Let $P=\{t_0,\dots,t_n\}$ be a partition of $[s,t]\subset[0,T]$ with $t_0=s$ and $t_n=t$. 
  Let $\kappa\in\R$. Let $p\geq 2$, $\eta>\frac12$, and $\bar{\eta}>1$. Suppose that $\Xi\in H_\eta^p(\cH_\kappa)\cap\bar{H}_{\bar{\eta}}^p(\cH_\kappa)$.  We have that
  \begin{align}
    \left\|\Expec{\sum_{i=1}^{n-1}S_{t_n-t_{i+1}}\Xi_{t_i,t_{i+1}}-\Xi_{t_0,t_n}\,\middle|\,\F_{t_0}}\right\|_{L^p(\cH_\kappa)}&\lesssim |t-s| ^{\bar\eta}, \label{eq:sewing_helper_1}\\
    \left\|\sum_{i=1}^{n-1}S_{t_n-t_{i+1}}\Xi_{t_i,t_{i+1}}-\Xi_{t_0,t_n}\right\|_{L^p(\cH_\kappa)}&\lesssim   |t-s|^{\eta}. \label{eq:sewing_helper_2}
  \end{align}
\end{lemma}

\begin{proof} 
Let $\Delta_k =\frac{t_n-t_0}{2^k}$ and consider sequences of points $ r_i^k\leq u_i^k\leq v_i^k\leq w_i^k$ in the  interval $[t_0+i\Delta_k, t_0+(i+1) \Delta_k]$, $i=0,\dots, 2^k-1$. 
 Let
 $$R_i^k=-S_{t_n-w_i^k}\big(S_{w_i^k-v_i^k}\; \hat\delta\Xi_{r_i^k,u_i^k, v_i^k}+\hat\delta\Xi_{r_i^k,v_i^k, w_i^k}\big).$$
Applying the assumption  $\Xi\in H_\eta^p(\cH_\kappa)\cap\bar H_{\bar\eta}^p(\cH_\kappa)$ to each term in the definition of $R_i^k$, we immediately have the
 trivial estimate
 $$  \big\|\mathbb{E}\big[R_i^k\,\big|\,\F_{r_i^k}\big]\big\|_{L^p}\lesssim (\Delta_k)^{\bar\eta}.$$
 Furthermore, 
  \begin{align*}
    \big\|R_i^k-\mathbb{E}\big[R_i^k\,\big|\,\F_{r_i^k}\big]\big\|_{L^p}&\lesssim 
     \big\|R_i^k\big\|_{L^p}\lesssim (\Delta_k)^{\eta}.
  \end{align*}
  We claim that we can calibrate $r_i^k\leq u_i^k\leq v_i^k\leq w_i^k$ in such a way that
  \begin{equation}\label{eq:dyadic_decompotion}
    \sum_{i=1}^{n-1}S_{t_n-t_{i+1}}\Xi_{t_i,t_{i+1}}-\Xi_{t_0,t_n}=\sum_{k=0}^\infty \sum_{i=0}^{2^k-1}R_i^k,
  \end{equation}
which then immediately leads to \eqref{eq:sewing_helper_1}:
  \begin{equation*}
    \left\|\Expec{\sum_{i=1}^{n-1}S_{t_n-t_{i+1}}\Xi_{t_i,t_{i+1}}-\Xi_{t_0,t_n}\,\middle|\,\F_{t_0}}\right\|_{L^p}\leq\sum_{k=0}^\infty\big\|\mathbb{E}\big[R_i^k\,\big|\,\F_{r_i^k}\big]\big\|_{L^p}\lesssim |t-s|^{\bar\eta}.
  \end{equation*}
  Similarly, the estimate (\ref{eq:bdg_consequence_2})  yields
  \begin{align*}
    \left\|\sum_{i=0}^{2^k-1}R_i^k\right\|_{L^p}
    &\lesssim\sum_{i=0}^{2^k-1}\big\|\mathbb{E}\big[R_i^k\,\big|\,\F_{r_i^k}\big]\big\|_{L^p} + \left(\sum_{i=0}^{2^k-1}\Big\|R_i^k-\mathbb{E}\big[R_i^k\,\big|\,\F_{r_i^k}\big]\Big\|_{L^p}^2\right)^{\frac12} \\
    &\lesssim 2^{k-k\bar{\eta}}|t-s|^\eta+2^{\frac{k-2k\bar \eta}{2}}|t-s|^{\bar\eta}\lesssim \frac{|t-s|^\eta}{2^{k[(\eta-1)\wedge(\eta-\frac12)]}}
  \end{align*}
  and \eqref{eq:sewing_helper_2} follows at once. 

  To conclude the proof, it remains to show \eqref{eq:dyadic_decompotion} holds. Let $Q=\{s_0,\dots,s_m\}$ be also a partition of $[s,t]$. We set 
  \begin{equation}\label{MQ}
    \M(Q) = \begin{cases}
    \sum_{i=0}^{m-1}S_{t_n-s_{i+1}}\Xi_{s_i,s_{i+1}}-S_{t_n-s_m}\Xi_{s_0,s_m}, & \text{if \#} Q>1\\
    0, &\text{otherwise.}
    \end{cases}
  \end{equation}
  We also need the following dyadic sub-partitions of $P$: Let $k\in \N_0$ and $P_i^k$ those points from $P$ falling into the $i^\text{th}$ sub-interval of the dyadic partition of level $k$:
  
  \begin{equation*}
    P_{i}^k\define P\cap\begin{cases}
    \left[i\Delta_k, (i+1)\Delta_k\right), & i=0,\dots, 2^k-2,\\
    \left[(2^k-1)\Delta_k, t_n\right], & i=2^k-1.
    \end{cases}
  \end{equation*}
  Note that each $P_i^\ell$ contains at most one point for $\ell\in\N$ sufficiently large.

   We define the random variables $\tilde R_i^k$  by
  \begin{equation}\label{eq:def_tilde_r}
    \tilde R_i^k\define\M\big(P_i^k\big)-\M\big(P_{2i}^{k+1}\big)-\M\big(P_{2i+1}^{k+1}\big).
  \end{equation}
   Since $P_i^k=P_{2i}^{k+1} \cup P_{2i+1}^{k+1}$,  $ \tilde R_i^k$ vanishes unless
 both $P_{2i}^{k+1}$ and $P_{2i+1}^{k+1}$ are non-empty. Since $P_0^0=P$, it holds that
 $$ \sum_{i=1}^{n-1}S_{t_n-t_{i+1}}\Xi_{t_i,t_{i+1}}-\Xi_{t_0,t_n}=\M(P_0^0).$$
  Also, rearranging the definition \eqref{eq:def_tilde_r} we have
  $$\M(P_0^0)=\M(P_0^1)+\M(P_1^1) + \tilde R_0^0.$$ 
  Iterating this identity shows that, for any $\ell\in \N$,
  \begin{equation*}
    \sum_{i=1}^{n-1}S_{t_n-t_{i+1}}\Xi_{t_i,t_{i+1}}-\Xi_{t_0,t_n} = \sum_{i=0}^{2^\ell-1}\M\big(P_i^\ell\big)+\sum_{k=0}^{\ell-1}\sum_{i=0}^{2^k-1}\tilde R_i^k.
  \end{equation*}
 Recall that, for $\ell$ sufficiently large, we have $\M\big(P_i^\ell\big)=0$.

  Finally, if either of $P_{2i}^{k+1}$ and $P_{2i+1}^{k+1}$ is empty, then $\tilde R_i^k=0$ so that we can choose $$r_i^k=u_i^k=v_i^k=w_i^k=i\Delta_k.$$
   Otherwise, we can pick
  \begin{equation*}
    r_i^k\define\min\big(P_{2i}^{k+1}\big),\quad u_i^k\define\max\big(P_{2i}^{k+1}\big), \quad v_i^k\define\min\big(P_{2i+1}^{k+1}\big),\quad w_i^k\define\max\big(P_{2i+1}^{k+1}\big).
  \end{equation*}
  In fact since $P_i^k=P_{2i}^{k+1}\cup P_{2i+1}^{k+1}$, for the numbers $r_i^k, u_i^k, v_i^k, w_i^k$ constructed earlier we have that 
  \begin{align*}
    \tilde R_i^k&=S_{t_n-v_i^k}\Xi_{u_i^k,v_i^k}+S_{t_n-u_i^k}\Xi_{r_i^k,u_i^k}+S_{t_n-w_i^k}\Xi_{v_i^k,w_i^k}-S_{t_n-w_i^k}\Xi_{r_i^k,w_i^k}\\
    &=-S_{t_n-w_i^k}\big(S_{w_i^k-v_i^k}\hat\delta\Xi_{r_i^k,u_i^k, v_i^k}+\hat\delta\Xi_{r_i^k,v_i^k, w_i^k}\big).
  \end{align*}
We define $R_i^k$ to be $\tilde R_i^k$ to conclude the construction of (\ref{eq:dyadic_decompotion}) and the proof of the proposition.
\end{proof}

Next, we present the main result of this section:
\begin{proposition}[Mild stochastic sewing lemma]\label{prop:stochastic_mild_sewing}
  Let $\kappa\in\R$. Let $p\geq 2$, $\eta>\frac12$, and $\bar{\eta}>1$. Suppose that $\Xi\in H_\eta^p(\cH_\kappa)\cap\bar{H}_{\bar{\eta}}^p(\cH_\kappa)$. Then, \begin{enumerate}
  \item for every $0\leq s\leq t\leq T$, the limit
    \begin{equation}\label{eq:sewing_limit}
        \I\Xi_{s,t}\define\lim_{|P|\to 0}\sum_{[u,v]\in P}S_{t-v}\Xi_{u,v}
    \end{equation}
    exists in $L^p(\Omega;\cH_\kappa)$ along any sequence of partitions $P$ of $[s,t]$ with mesh tending to zero and the process $\I\Xi$ is additive in the following sense:
\begin{equation*}
S_{t-s}\I\Xi_{r,s}+\I\Xi_{s,t}=\I\Xi_{r,t} \qquad \forall\, (r,s,t)\in\Delta^3.
\end{equation*}
Moreover $\I\Xi$ vanishes if $\|\Xi_{s,t}\|_{L^p(\cH_\kappa)}\lesssim|t-s|^{\bar{\eta}}$.
 \item  If, in addition, for any $(r,s,t)\in\Delta^3$, there exists $\Lambda$ such that 
     $$\hat\delta \Xi_{r,s,t}=S_{t-s}\Lambda_{r,s,t}$$ and such that 
    \begin{equation} \label{Lambda-bound}
      \big\|\Lambda_{r,s,t}\big\|_{L^p(\cH_{\kappa})}\leq C|t-r||t-s|^{\bar{\eta}-1},
    \end{equation}
    then, for any $\varpi\in[0,\bar\eta)$,
    \begin{equation}\label{eq:mild_stochastic_sewing}
        \big\|\I\Xi_{s,t}-\Xi_{s,t}\big\|_{L^p(\cH_{\kappa+\varpi})}\lesssim C|t-s|^{\bar{\eta}-\varpi}
    \end{equation}
    for all $0\leq s\leq t\leq T$. 
  \end{enumerate}\end{proposition} 

\begin{proof}
  \dashuline{\textbf{Step 1: Existence.}}  
  We first show the existence of the limit \eqref{eq:sewing_limit} in $L^p$. Let $P=\{t_i\}_{i=1}^{n}$ and $P^\prime=\{t_i^\prime\}_{i=1}^{n^\prime}$  be two partitions of $[s,t]$ and  $Q\define P\cup P^\prime=\{s_i\}_{i=1}^{m}$ their mutual refinement.
    Denote 
    $$\Xi^P=\sum_{[u,v]\in\mathcal{P}}S_{t-v}\Xi_{u,v}$$ 
    and similar for the other partitions. Define 
    $$
    Z_i=\sum_{j:s_j\in[t_i,t_{i+1})}S_{t-s_{j+1}}\Xi_{s_j,s_{j+1}}-S_{t-t_{i+1}}\Xi_{t_i,t_{i+1}}
    $$
  so that
  \begin{equation*}
    \Xi^{Q}-\Xi^P=\sum_{i=1}^{n-1}Z_i.
  \end{equation*}
  The inequality \eqref{eq:bdg_consequence_2} and \cref{lem:sewing_help} applied to each sub-interval $[t_i,t_{i+1}]$, therefore show that
  \begin{align*}
    \big\|\Xi^Q-\Xi^P\big\|_{L^p(\cH_\kappa)}&\lesssim\sum_{i=1}^{n-1}\big\|\Expec{Z_i\,|\,\F_{t_i}}\big\|_{L^p}+\left(\sum_{i=1}^{n-1}\|Z_i\|_{L^p}^2\right)^{\frac12}\\
    &\lesssim\sum_{i=1}^{n-1}|t_{i+1}-t_i|^{\bar\eta}+\left(\sum_{i=1}^{n-1}|t_{i+1}-t_i|^{2\eta}\right)^{\frac12}\lesssim |P|^{\rho},
  \end{align*}
  where $\rho\define (\bar\eta-1)\wedge (\eta-\frac12)>0$. Similarly we find $\|\Xi^Q-\Xi^{P^\prime}\|_{L^p(\cH_\kappa)}\lesssim |P^\prime|^{\rho}$, whence
  \begin{equation*}
    \|\Xi^P-\Xi^{P^\prime}\|_{L^p(\cH_\kappa)}\lesssim \big(|P|\vee |P^\prime|\big) ^{\rho}.
  \end{equation*}
  Consequently, the sequence $(\Xi^{P_n})_n$ is Cauchy in $L^p$ along any sequence of partitions $P_n$ of $[s,t]$ with mesh $|P_n|\to 0$, and the existence of $\I\Xi_{s,t}$ follows. Note that the asserted `additivity' $S_{t-u}\Xi^P_{s,u}+\Xi^P_{u,t}=\Xi^P_{s,t}$ holds for each partition $P$. It passes immediately to the limit $\I\Xi$ from this construction.

  \dashuline{\textbf{Step 2: Uniqueness.}} Suppose that $\|\Xi_{s,t}\|_{L^p(\cH_\kappa)}\lesssim |t-s|^{\bar\eta}$. We show that then $\I\Xi=0$. This is in fact an immediate consequence of the a priori $L^p$-bound:
  \begin{equation*}
    \|\I\Xi_{s,t}\|_{L^p(\cH_\kappa)}\leq\lim_{|P|\to 0}\sum_{[u,v]\in P}\|\Xi_{u,v}\|_{L^p(\cH_\kappa)}\lesssim \lim_{|P|\to 0}|P|^{\bar\eta-1}=0.
  \end{equation*}

  \dashuline{\textbf{Step 3: The bound \eqref{eq:mild_stochastic_sewing}.}} We consider an approximation of $\I\Xi_{s,t}$ along the dyadic rationals in the interval $[s,t]$. Recall our notation 
  $\bD_n=\{s+k\frac{t-s}{2^n}:\,k=0,\dots,2^n\}$. Then we have
   \begin{align*}
    \Xi^{\bD_{n+1}}_{s,t}&=\sum_{[u,v]\in \bD_{n+1}}S_{t-v}\Xi_{u,v}\\
    &=\Xi^{\bD_{n}}_{s,t}-\sum_{[u,w]\in \bD_n}S_{t-w}
 \Big(\Xi_{u,w} -\Xi_{\f {w+u} 2,w} -S_{w-\f {w+u} 2}\Xi_{u,\f {w+u} 2} \Big) \\
 &=\Xi^{\bD_{n}}_{s,t}-\sum_{[u,w]\in \bD_n}S_{t-w}\hat\delta\Xi_{u,\frac{w+u}{2}, w}.
 \end{align*}
  Hence, we find 
  \begin{align*}
    \big\|\Xi^{\bD_{n+1}}_{s,t}-\Xi^{\bD_{n}}_{s,t}\big\|_{L^p(\cH_{\kappa+\varpi})}
       &=\bigg\|\sum_{[u,w]\in \bD_n}S_{t-w}\hat\delta\Xi_{u,\frac{w+u}{2}, w}\bigg\|_{L^p(\cH_{\kappa+\varpi})}\\
          &=\bigg\|\sum_{[u,w]\in \bD_n} S_{t-\frac{w+u}{2}}\Lambda _{u,\frac{w+u}{2}, w}\bigg\|_{L^p(\cH_{\kappa+\varpi})}.
  \end{align*}
 Applying the smooothing property of $(S_t)_{t\geq 0}$,
  $\|S_t x\|_{\cH_{\kappa+\varpi}}\lesssim t^{-\varpi}\| x\|_{\cH_{\kappa}}$ (see \eqref{eq:semigroup_interpolation}), we have
    \begin{align*}
      \big\|\Xi^{\bD_{n+1}}_{s,t}-\Xi^{\bD_{n}}_{s,t}\big\|_{L^p(\cH_{\kappa+\varpi})}
      &\lesssim \sum_{[u,w]\in \bD_n} \left(t-\frac{w+u}{2}\right)^{-\varpi}\Big\|\Lambda _{u,\frac{w+u}{2}, w}\Big\|_{L^p(\cH_{\kappa})}\\
     &\lesssim \frac{C (t-s)^{\bar \eta}}{2^{n+(n+1)(\bar \eta-1)}}\sum_{[u,w]\in \bD_n}\left(t-\frac{w+u}{2}\right)^{-\varpi}\\
        &\lesssim \frac{C (t-s)^{\bar \eta}}{2^{n+(n+1)(\bar \eta-1)}}
      \sum_{[u,w]\in \bD_n}\left(t-\frac{w+u}{2}\right)^{-\varpi}.
  \end{align*}
 Since $(t-\frac{w+u}{2})\geq \f{t-s}{2^{n+1}}$, for each $\epsilon\in\big((\varpi-1)_+,\bar\eta-1\big)$,
  \begin{align*}
    \sup_{n\in\N}\frac{t-s}{2^n}\left(\frac{t-s}{2^{n+1}}\right)^\epsilon\sum_{[u,w]\in \bD_n}\left(t-\frac{w+u}{2}\right)^{-\varpi}&\leq\sup_{n\in\N}\frac{t-s}{2^n}\sum_{[u,w]\in \bD_n}\left(t-\frac{w+u}{2}\right)^{-\varpi+\epsilon}\\
    &\leq\int_s^t (t-r)^{-\varpi+\epsilon}\,dr\lesssim (t-s)^{1-\varpi+\epsilon},
  \end{align*}
  where we used the fact that the Riemann sum increases monotonically to the integral. In particular, we see that
  \begin{equation*}
    \big\|\Xi^{\bD_{n+1}}_{s,t}-\Xi^{\bD_{n}}_{s,t}\big\|_{L^p(\cH_{\kappa+\varpi})}
    \lesssim\frac{C (t-s)^{\bar \eta-\bar \varpi}}{2^{(n+1)(\bar \eta-1-\epsilon)}}.
  \end{equation*}
  Since $\bar \eta-1-\epsilon>0$, this gives
  \begin{equation*}
    \big\|\I\Xi_{s,t}\big\|_{L^p(\cH_{\kappa+\varpi})}\leq\sum_{n=0}^\infty\big\|\Xi^{\bD_{n+1}}_{s,t}-\Xi^{\bD_{n}}_{s,t}\big\|_{L^p(\cH_{\kappa+\varpi})}\lesssim C(t-s)^{\bar\eta-\varpi},
  \end{equation*}
  as required.
 \end{proof}
\begin{remark}
After putting the article on the arXiv,  we learnt from Khoa L\^e a nice trick to deduce part (i)  of \Cref{prop:stochastic_mild_sewing}
 directly from the stochastic sewing lemma  in \cite{Le-sewing-banach-21}, where the usual incremental process is used.
To see this let us fix a terminal time $t>0$ and for $u<v<w\le t$, and set
$$A_{u,v}=S_{t-v}\Xi _{u,v}.$$
Then $$\delta A_{u,v, w}=S_{t-w}(\Xi_{u,w}-\Xi_{v,w}-S_{t-v}\Xi_{u,v})=S_{t-w} \hat\delta \Xi_{u,v,w}.$$
Naturally, $\Xi\in H_\eta^p(\cH_\kappa)\cap\bar{H}_{\bar{\eta}}^p(\cH_\kappa)$ implies that 
\begin{align*}
 \sup_{0\leq s<t\leq T}\frac{\|A_{uv}\|_{L^p}}{|u-v|^\eta}<\infty,  \qquad
\frac{\|\expec{\delta A_{sut}|\mathcal{F}_s}\|_{L^p}}{|s-t|^{\bar{\eta}}}<\infty,
\end{align*}
and \cite[Theorem A]{Le-sewing-banach-21} allows to conclude that $\I\Xi_{s,t}=\lim \sum_{[u,v]}S_{t-v}\Xi_{u,v}$.
The same trick does not seem to deduce the error bound in \Cref{prop:stochastic_mild_sewing} from the corresponding estimate in
\cite[Theorem 4.1]{Athreya-Le-Butkovsky-Mytnik}.  The conditions in that article appear to be of more deterministic nature. The `seesaw' allowing us to trade interpolation space regularity for time regularity is crucial for the application in the sequel.
 Let $s<u<w \le t$. In  \cite[Theorem 4.1]{Athreya-Le-Butkovsky-Mytnik} two conditions are imposed, the first of  which is:
$\|\E[\delta A_{s,u,w}|\F_s]\|_{L^p} \lesssim u^{-\alpha_1}|t-u|^{-\beta_1}|w-s|^{1+\epsilon_1}$ where $\alpha_1, \beta_1\in [0, 1)$. 
The second  is on the conditional norm of the $L_p$ norm of $\delta A_{s,u,w}$ with constants $\alpha_2, \beta_2$ restricted to $ [0, \f12)$---note no analogous condition is needed here. Then the authors obtained an estimate of the order $ |t-s|^{\bar \eta-1}$. In \Cref{eq:mild_stochastic_sewing} the estimate is of the order
$|t-s|^{\bar{\eta}-\varpi}$ with $\varpi \in [0, \bar \eta)$ and $\bar \eta>1$. In addition, given the restriction  $\alpha_1, \beta_1\in [0, 1)$, it does not seem possible to deduce the aforementioned condition on $\|\E[\delta A_{s,u,w}|\F_s]\|_{L^p}$ from our bound $ \big\|\Lambda_{r,s,t}\big\|_{L^p(\cH_{\kappa})}\leq C|w-s||w-u|^{\bar{\eta}-1}$. In fact, from the inequality
\begin{equation}
\big\|\delta A_{s,u,w}\big\|_{\cH_{\kappa+\alpha}}\lesssim \|S_{t-w}\|_{L(\cH_\kappa, \cH_{\kappa+\alpha})} \|\delta \Xi_{x, u,w}\|_{\cH_\kappa}
\lesssim |t-w|^{-\alpha}\big\|\hat\delta \Xi_{x, u,w}\big\|_{\cH_\kappa},
\end{equation}
one has only
$\|\delta A_{s,u,w}\|_{L^p(\cH_{\kappa+\varpi})}\lesssim  |t-w|^{-\varpi}|w-s||w-u|^{\bar \eta -1}\lesssim  |t-w|^{-\varpi}|w-s|^{\bar \eta}$.
One can attempt to use $\alpha_1=0$, $\beta_1=\varpi$ and $\epsilon_1=\bar \eta -1$. However $\varpi$ may fall out of the permissible range of $\beta_1$. 
\end{remark}

\section{Uniform $L^p$-Estimates}\label{sec:lp_estimates}

The aim of this section is to leverage the mild stochastic sewing lemma (\cref{prop:stochastic_mild_sewing}) to obtain strong $L^p$-bounds on the SPDE
\begin{equation}\label{eq:spde}
  dx_t=\big(Ax_t+f(t,x_t)\big)\,dt+g(t,x_t)\,dB_t,
\end{equation}
where $B$ is a trace-class fBm, see \cref{sec:fbm} for details. We emphasize that these bounds do \emph{not} follow from the pathwise approach taken in \cref{sec:well_posed}.

\subsection{Trace-Class Fractional Brownian Motions}\label{sec:fbm}
 A one-dimensional fractional Brownian motion (fBm) with Hurst parameter $H\in(0,1)$ is the centered Gaussian process $(\beta_t)_{t\geq 0}$ with covariance
\begin{equation*}
  \Expec{\beta_s \beta_t}=\frac12\left(t^{2H}+s^{2H}-|t-s|^{2H}\right)\qquad\forall\,s,t\geq 0.
\end{equation*}
For $H=\frac12$, this is the standard Wiener process.

Let $Q\in L(\cK)$ be a symmetric, non-negative trace-class operator on a separable Hilbert space $\cK$.  Recall a trace-class operator is a compact. Let $\{\lambda_n\}_{n\in\N}\subset\R_+$ be its discrete family of eigenvalues, counted with multiplicity. The associated normalized eigenvectors $Qe_n=\lambda_n e_n$ form an orthonormal basis of $\cK$. We refer the readers to the classical book by Da Prato and Zabczyk \cite{DaPrato1992}. Albeit the interest there lies in Markovian systems, we follow their terminology.

  \begin{definition}\label{def:trace_class_fbm} 
 Let $Q:\cK\to\cK$ be a symmetric, non-negative definite trace-class operator. An $\cK$-valued centred Gaussian process $(B_t)_{t\geq 0}$ is called a \emph{fractional Brownian motion (fBm)} with covariance $Q$ and Hurst parameter $H$ if 
  \begin{equation*}
    \Expec{\braket{B_s, x}\braket{B_t,y}}=\frac{1}{2}\left(t^{2H}+s^{2H}-|t-s|^{2H}\right)\braket{Qx,y}\qquad\forall\,s,t\geq 0\quad\forall\,x,y\in\cK.
  \end{equation*}
\end{definition}

\begin{remark}\label{trace}
In the study of stochastic partial differential equations driven by Q-fBms, it is often assumed that $\sqrt{Q}$ is trace-class. See e.g. Garrido-Atienza-Lu-Schmalfu{\ss} \cite{Garrido2010}, Pei-Xu-Bai \cite{Pei2020},  Nascimento-Ohashi \cite{Nascimento2021}, and \cite{Maslowski-Nualart}. We emphasize that we do \emph{not} impose this stronger condition.
\end{remark}

Set $\beta_t^n=\f 1{ \sqrt{\lambda_n}} \<B_t, e_n\>$. Then, $\{\beta^n\}$ are independent one-dimensional fBm's with parameter $H$. Conversely, for any orthonormal basis $\{e_n\}_{n\in\N}\subset\cK$ and any sequence $(\beta^n)_{n\in\N}$ of i.i.d. one-dimensional fBms, the series
\begin{equation}\label{eq:trace_class_fbm}
  B_t=\sum_{n=1}^\infty \sqrt{\lambda_n}\beta^n_t e_n,\qquad t\geq 0,
\end{equation}
which converges almost surely and in $L^p$ ($p\geq 1$), is a $Q$-fBm:
\begin{lemma}\label{lem:uniform_convergence}
A trace-class fBm $(B_t)_{t\in[0,T]}$ with Hurst parameter $H$ takes values in $\C^{H-}\big([0,T],\cK\big)$ with probability one. 

   Let $(\beta_t^n)$ be  i.i.d. one-dimensional centred Gaussian processes with continuous sample paths with finite variance with $t\in [0,T]$.  Then, for any sequence $(\lambda_n)\in\ell^1(\R_+)$ and any orthonormal system $(e_n)_{n\in\N}$ of $\cK$, the series 
$
  \sum_{n=1}^\infty\sqrt{\lambda_n}\beta_t^n e_n$ 
   converges in $\C\big([0,T],\cK\big)$ $\prob$-almost surely. 
\end{lemma} 
The  first statement is a straightforward  application of the Kolmogorov continuity theorem, while the second follows from the convergence of
$ \sup_{t\le T} \sum_{k=n}^m\sqrt{\lambda_k}\beta_t^n e_k$ for any $T>0$ in $L^2$. This in turn is due to the fact that $\big\|\|\sum_{n=1}^\infty \sqrt{\lambda_n}\beta^n e_n\|_\infty\big\|_{L^2} 
\le  \Expec{|\beta^1|_\infty^2}\sum_n\lambda_n<\infty$.

\subsection{Basic Properties of Integrals Against Trace-Class fBm}
Our analysis in the sequel is underpinned by the following Mandelbrot-van Ness representation of the one-dimensional fBm with Hurst parameter $H\in(0,1)$ \cite{Mandelbrot1968}:
\begin{equation}\label{eq:mandelbrot}
  \beta_t=\alpha_H\left(\int_{-\infty}^0 (t-u)^{H-\frac12}-(-u)^{H-\frac12}\,dW_u-\int_0^t(t-u)^{H-\frac12}\,dW_u\right).
\end{equation}
Here $\alpha_H>0$ is an explicitly known normalization constant. By virtue of \eqref{eq:mandelbrot}, we have the following \emph{locally independent} decomposition of the one-dimensional fBm increment \cite{Hairer2005,Li2020}. For $t,h\geq 0$:
\begin{equation}\label{eq:increment_decomposition}
  \beta_{t+h}-\beta_t=\alpha_H\int_{-\infty}^t (t+h-u)^{H-\frac12}-(t-u)^{H-\frac12}\,dW_u+\alpha_H\int_t^{t+h}(t+h-u)^{H-\frac12}\,dW_u\define\bar\beta_h^t+\tilde\beta_h^t.
\end{equation}
If $t=0$, we shall write $\bar\beta_h\define \bar\beta_h^0$ and similarly $\tilde\beta_h\define\tilde \beta_h^0$. 
For $h>0$, there is no singularity in the integrand of the first term, whence an integration by parts formula turns it into a Riemann integral. This is exploited in the next lemma:
\begin{lemma}\label{lem:one_dim_smooth_part}
  For each $t\geq 0$, the mapping $(0,\infty)\ni h\mapsto\bar{\beta}_h^t$ is smooth with probability one and we have
  \begin{align}
  \dot{\bar\beta}^t_h\define\frac{d}{dh}\bar\beta_h^t&=\alpha_H\left(H-\frac12\right)\int_{-\infty}^t(t+h-u)^{H-\frac32}\,dW_u, \label{eq:first_derivative_smooth}\\
  \ddot{\bar\beta}^t_h\define\frac{d^2}{dh^2}\bar\beta_h^t&=\alpha_H\left(H-\frac12\right)\left(H-\frac32\right)\int_{-\infty}^t(t+h-u)^{H-\frac52}\,dW_u \label{eq:second_derivative_smooth}
\end{align}
for each $h>0$.
\end{lemma}
\begin{proof}
  We only demonstrate the computation of the first derivative as the second one is similar. Fix $h>0$ and $t\geq 0$. Let us begin with an integration by parts:
  \begin{align}
    \alpha_H^{-1}\bar{\beta}_h^t&=\left((t+h-u)^{H-\frac12}-(t-u)^{H-\frac12}\right)(W_u-W_t)\Big|_{u=-\infty}^t \label{eq:boundary_term_derivative}\\
    &\phantom{=}+\left(H-\frac12\right)\int_{-\infty}^t \left((t+h-u)^{H-\frac32}-(t-u)^{H-\frac32}\right)(W_u-W_t)\,du. \nonumber
  \end{align}
  Note that there is an almost surely finite random variable $C>0$ such that
  \begin{equation*}
    |W_u-W_t|\leq \begin{cases}
    C\sqrt{|u-t|}, & u\in[t-1,t],\\
    C\big(|u|^{\frac{1+H}{2}}+1\big), & u\in(-\infty, t-1).
    \end{cases}
  \end{equation*}
  In particular, the boundary term \eqref{eq:boundary_term_derivative} vanishes. Moreover, we may differentiate under the integral sign to obtain
  \begin{align*}
    \alpha_H^{-1}\dot{\bar{\beta}}_h^t=\alpha_H^{-1}\frac{d}{dh}\bar{\beta}_h^t&=\left(H-\frac12\right)\left(H-\frac32\right)\int_{-\infty}^t(t+h-u)^{H-\frac52}(W_u-W_t)\,du \\
    &=-\left(H-\frac12\right)(t+h-u)^{H-\frac32}(W_u-W_t)\Big|_{u=-\infty}^t \\
    &\phantom{=}+\left(H-\frac12\right)\int_{-\infty}^t(t+h-u)^{H-\frac32}\,dW_u.
  \end{align*}
  Again, the boundary term vanishes and \eqref{eq:first_derivative_smooth} follows at once.
\end{proof}

The increment decomposition \eqref{eq:increment_decomposition} gives rise to a similar decomposition of the trace-class fBm \eqref{eq:trace_class_fbm}:
\begin{equation*}
  B_{t+h}-B_t=\sum_{n=1}^\infty\sqrt{\lambda_n}\bar\beta_h^{n,t} e_n + \sum_{n=1}^\infty\sqrt{\lambda_n}\tilde\beta_h^{n,t} e_n\define\bar B_h^t + \tilde B_h^t.
\end{equation*}
Once again, these series converge almost surely and in $L^p(\Omega;\cK)$ for any $p\geq 1$. 
\begin{lemma}\label{lem:derivative_smooth_part}
  Let $t\geq 0$. The mapping $h \in (0,\infty)\mapsto\bar{B}_h^t$ is almost surely smooth. Moreover, for each $p\geq 1$, we have that
  \begin{equation*}
    \|\dot{\bar B}^t_h\|_{L^p}\lesssim \frac{1}{h^{1-H}},\qquad \|\ddot{\bar B}^t_h\|_{L^p}\lesssim \frac{1}{h^{2-H}},
  \end{equation*}
  uniformly in $t\geq 0$ and $h\geq 0$.
\end{lemma}
\begin{proof}
  Fix $t\geq 0$. Let $h\ge 0$.  We claim that 
  \begin{equation*}
    \dot{\bar{B}}_h^t=\sum_{n=1}^\infty\sqrt{\lambda_n}\dot{\bar{\beta}}_h^{n,t} e_n.
  \end{equation*}
  To see this, for any $n\in \N$, let us abbreviate $\fS_h^N =\sum_{n=1}^N\sqrt{\lambda_n}\bar{\beta}_h^{n,t} e_n$. By \cref{lem:one_dim_smooth_part}, we know that $\frac{d}{dh}\fS_h^N=\sum_{n=1}^N \sqrt{\lambda_n}\dot{\beta}_h^{n,t} e_n$ for each $h>0$. Thanks to \cref{lem:uniform_convergence}, we see that the partial sums $\fS^N$ and $\frac{d}{dh}\fS_h^N$ converge $\prob$-a.s. locally uniform on $(0,\infty)$ to $\bar{B}^t$ and $\dot{\bar B}^t$, respectively. In particular, $h\mapsto\bar{B}^t_h$ is almost surely differentiable and $\frac{d}{dh}\bar{B}^t_h=\dot{\bar B}_h^t$ for each $h>0$.

  For the $L^p$-bound, we note $\dot{\bar{B}}_h^t$ is centered Gaussian with variance $\sum_n \lambda_n\expec{\big(\dot{\bar\beta}_h^t\big)^2}$. By \eqref{eq:first_derivative_smooth}, we find
  \begin{equation*}
    \expec{\big(\dot{\bar\beta}_h^t\big)^2}\lesssim\int_{-\infty}^0 (h-u)^{2H-3}\,du\lesssim h^{2H-2}
  \end{equation*}
  and the bound on $\|\dot{\bar B}^t_h\|_{L^p}$ follows by Gaussianity.
  The second derivative can be handled similarly using \eqref{eq:second_derivative_smooth}.
\end{proof}

Let $(\F_t^B)_{t\geq 0}$ be the filtration generated by the trace-class fBm.  Let $g:[s,t]\to L(\cK,\cH)$ be $\F_s$-measurable and write $g_n\define g e_n$ with the orthonormal basis in the series representation of $B$ \eqref{eq:trace_class_fbm}.  We  define a \emph{mixed Wiener-Young integral} against $B$ (which is a priori different than the Young integral):
\begin{align}
  \int_s^t g(r)\,\d B_r&\define\int_s^t g(r)\,\tilde{B}_{r-s}^s+\int_s^t g(r)\,d\bar{B}_{r-s}^s\nonumber \\
  &=\sum_{n=1}^\infty\sqrt{\lambda_n}\left(\int_s^t g_n(r)\,d \tilde \beta_{r-s}^{n,s}+\int_s^t g_n(r)\dot{\bar\beta}_{r-s}^{n,s}\,dr\right), \label{eq:mixed_integral}
\end{align}
where the first integral is the Wiener integral against the Gaussian process $(\tilde{\beta}_r^{n,s})_{r\geq 0}$.
Set
$$  \K_s(r,s)=\Expec{\tilde \beta_r^{1}\tilde \beta_s^{1}}.$$

\begin{lemma}\label{lem:wiener_integral}
The Wiener integral $\int_s^t g(r)\,d\tilde{B}_{r-s}^s$
  is a centered Gaussian process with covariance operator $Q_g$ given by
  \begin{equation*}
    \Braket{x,Q_gy}\define\sum_{n=1}^\infty\lambda_n\int_s^t\int_s^t \Braket{g_n(v),x}\Braket{g_n(u),y}\frac{\partial^2}{\partial u\partial v}\K(u-s,v-s)\,du\,dv,\qquad \forall\,x,y\in\cH.
  \end{equation*}
  \end{lemma}
  \begin{proof}
  This follows from 
  $  \int_s^t g(r)\,d  \tilde B_r = \sum_{n=1}^\infty\sqrt{\lambda_n}\int_s^t g_n(r)\,d \tilde \beta_{r-s}^{n,s}$
 and the It\^o isometry \cite{Huang1978},  
 \begin{align*}
  &  \Expec{\left(\int_s^t \Braket{g_n(r),x}\,d\tilde{\beta}_{r-s}^{n,s}\right)\left(\int_s^t \Braket{g_n(r),y}\,d\tilde{\beta}_{r-s}^{n,s}\right)}\\
  &=\int_s^t\int_s^t \braket{g_n(u),x}\braket{g_n(v),y}\frac{\partial^2}{\partial u\partial v}\K(u-s,v-s)\,du\,dv. \qedhere
  \end{align*}
\end{proof}
  
Let $|G|_{\C^{-\delta}}=\sup_{0\le s\le t\le T} |s-t|^{\kappa-1} \big\|\int_s^t g(r)dr\big\|_\cE$ denote the  H\"older norm of a function $G: [0,T]\to \cE$ of negative exponent $-\delta$ where $\cE$ is a Banach space.

\begin{lemma}\label{lem:lp_bound}
  Let $0\leq s\leq t\leq T$ and $\delta\in[0,1)$. Let $g\in\C^{-\delta}\big([s,t],L(\cK,\cH)\big)$ be $\F_s$-measurable. Then, for each $1\leq p<q$,
  \begin{equation*}
    \left\|\int_s^t g(r)\,\d B_r\right\|_{L^p}\lesssim\big\||g|_{\C^{-\delta}}\big\|_{L^q}|t-s|^{H-\delta}
  \end{equation*} 
  with a prefactor uniform in $s,t\in[0,T]$.
\end{lemma}
\begin{proof}
  Let us first assume that $g\in\C^{-\delta}$ is smooth. We prove the asserted $L^p$-bound for both integrals in \eqref{eq:mixed_integral} individually. Let us begin with the smooth part: An integration by parts shows that
  \begin{equation*}
    \int_s^t g(r)\, d\bar{B}_{r-s}^s=\left(\int_s^t g(r)\,dr\right) \dot{\bar B}_{t-s}^s-\int_s^t\left(\int_s^r g(u)\,du\right)\ddot{\bar{B}}_{r-s}^{s} \,dr.
  \end{equation*}
  In particular, using \cref{lem:derivative_smooth_part} we obtain
  \begin{equation*}
    \left\|\int_s^t g(r)\,d\bar{B}_{r-s}^s\right\|_{L^p}\lesssim\big\||g|_{\C^{-\delta}}\big\|_{L^q} \left(|t-s|^{H-\delta}+\int_s^t |r-s|^{\delta+H-1}\,dr\right)\lesssim\big\||g|_{\C^{-\delta}}\big\|_{L^q} |t-s|^{H-\delta}.
  \end{equation*}
  For the rough part, we simply notice that the Wiener integral satisfies
  \begin{equation*}
    \int_s^t g(r)\,d\tilde{B}_{r-s}^s=\sum_{n=1}^\infty \sqrt{\lambda_n}\int_s^t g_n(r)\,d\tilde{\beta}_{r-s}^{n,s}.
  \end{equation*}
  Since  $(\tilde{B}_h^s)_{h\in(0,t-s]}$ is independent of the stochastic processes $\{g(r)\}_{r\in (s, t]}$, the integral is a Gaussian variable.  It suffices to estimate the $L^2$-norm. By \cref{lem:wiener_integral} we obtain
  \begin{equation*}
    \left\|\int_s^t g(r)\,d\tilde{B}_{r-s}^s\right\|_{L^2}^2=\sum_{n=1}^\infty\lambda_n\Expec{\int_s^t\int_s^t \Braket{e_n,g(v)g^*(u)e_n}\frac{\partial^2}{\partial u\partial v}\K(u-s,v-s)\,du\,dv}.
  \end{equation*}
  It follows from \cite[Lemma A.1]{Hairer2020}, c.f. also Lemma 3.4 there, 
  \begin{equation*}
    \int_s^t\int_s^t \Braket{e_n,g(v)g^*(u)e_n}\frac{\partial^2}{\partial u\partial v}\K(u-s,v-s)\,du\,dv\lesssim |t-s|^{2(H-\delta)}|g_ng_n^*|_{\C^{-\delta}}.
  \end{equation*}
  Notice that $|g^*_n|_{\C^{-\delta}}\leq|g|_{\C^{-\delta}}$, whence
  \begin{equation*}
    \left\|\int_s^t g(r)\,d\tilde{B}_{r-s}^s\right\|_{L^2}^2\lesssim|t-s|^{2(H-\delta)}\big\||g|_{\C^{-\delta}}\big\|_{L^2}^2\sum_{n=1}^\infty\lambda_n\lesssim|t-s|^{2(H-\delta)}\big\||g|_{\C^{-\delta}}\big\|_{L^2}^2,
  \end{equation*}
completing the proof.  
\end{proof}

\subsection{$L^p$-Estimates and Stochastic Stability}
Let $(B_t)_{t\geq 0}$ be a trace-class fBm with Hurst parameter $H>\frac12$, driving the SPDE \eqref{eq:spde},  and $(\F_t^B)_{t\geq 0}$ is its filtration. Let $\cE$ and $\cK$ be separable Hilbert spaces and let
$$g: \Omega\times[0,T]\times \cE\to L(\cK,\cE)$$
be a random field with values in the space of linear operators $L(\cK,\cE)$, independent of $\F_T^B$. Set $\F_t^g=\sigma\big(g(s,x),\,s\leq t,\,x\in\cE\big)$.
For $\gamma\in(0,1]$, $\delta\in[0,1)$ and $p\geq 1$ we define the `norm-like' quantity
\begin{equation}\label{eq:a_norm}
  \vertiii{g}_{\A_{-\delta,\gamma}^p}\define\sup_{s\leq T}\left(\sup_{X\in\F_s^B\vee\F_s^g}\big\||g(\bigcdot,X)|_{\C^{-\delta}([s,T])}\big\|_{L^p}+\sup_{X,Z\in\F_s^B\vee\F_s^g}\frac{\big\||g(\bigcdot,X)-g(\bigcdot,Z)|_{\C^{-\delta}([s,T])}\big\|_{L^p}}{\big\|\|X-Z\|_\cE^\gamma\big\|_{L^p}}\right).
\end{equation}
Since it  depends on the filtration generated by $g$, the triangle inequality may not hold.  In the sequel, we only use the triangle inequality coming from the
$\C^{-\delta}([s,T])$ norm.
Eventually, we  take the Hilbert space $\cE=\cH_{\kappa+\alpha}$, 
the random field $(t,x)\mapsto g\big(x,Y_{\frac{t}{\varepsilon}}\big)$, and $\F_t^g=\F_{\f t \epsilon}^Y$.

\medskip

Let $X$ be the solution to the SPDE \eqref{eq:slow_spde} interpreted in the pathwise mild sense of \cref{pathwise}. Our aim is to derive $L^p$-estimates on the integral
\begin{equation*}
  \int_0^{\bigcdot} g(t, X_t)\,dB_t.
\end{equation*}
We have the following consequence of \cref{prop:stochastic_mild_sewing} in which we usually choose $\kappa=-\alpha$.

Let $(\F_t)_{t\geq 0}$ be a filtration on a probability space $(\Omega, \F, \prob)$ with the usual assumptions. For $\alpha\in(0,1]$ and $p\geq 1$ we define
\begin{equation}\label{B-alpha-p}
  \hat\B_{\alpha,p}\define\left\{Z:[0,T]\times\Omega\to\cH:\,Z\text{ is }(\F_t)_{t\geq 0}\text{-adapted},\,\|Z\|_{\hat\B_{\alpha,p}}\define\sup_{s\neq t}\frac{\big\|\hat\delta Z_{s,t}\big\|_{L^p}}
 {|t-s|^{\alpha}}<\infty\right\}.
\end{equation}
As before, we may also use the longer notation $ \hat\B_{\alpha,p}\big([0,T],\cH\big)$. Elements of $ \hat\B_{\alpha,p}$ are $L^p$-bounded:
 one has the bound $\|Z_t\|_{L^p}\lesssim \|Z_s\|_{L^p}+|t-s|^\alpha \|Z\|_{\hat\B_{\alpha,p}}$. We furthermore have the following mild version of the Kolmogorov continuity theorem:

\begin{proposition}[Mild Kolomogorov criterion]\label{prop:mild_kolmogorov}
  Let $(Z_t)_{t\in[0,T]}$ be an $\cH$-valued stochastic process such that  for some $\alpha,p>0$,
  \begin{equation*}
    \big\|\hat\delta Z_{s,t}\big\|_{L^p}\lesssim|t-s|^\alpha\qquad\forall\,s,t\in[0,T].
  \end{equation*}
  Then, for any $\gamma<\alpha-\frac{1}{p}$, $Z\in\hat\C^{\gamma}\big([0,T],\cH\big)$ (up to modification) and $\big\|\|Z\|_{\hat\C^\gamma}\big\|_{L^p}<\infty$. In other words,
  \begin{equation*}
    \hat\B_{\alpha,p}\hookrightarrow L^p\Big(\Omega,\hat\C^{(\alpha-\frac1p)-}\Big).
  \end{equation*}
\end{proposition}
\begin{proof}
  The argument closely resembles the classical proof of the Kolmogorov criterion. We shall include it for the reader's convenience.

  For each $n\in\N$ let $\bD_n\define\big\{\frac{kT}{2^n}:\,k=0,\dots,2^n\}$ be the dyadic rationals of order $n$ in $[0,T]$. We also set $\bD\define\bigcup_{n\in\N}\bD_n$. Fix $\gamma<\alpha-\frac1p$. By Markov's inequality, we have
  \begin{equation*}
    \prob\Big(\max_{k=1,\dots,2^n}\big\|\hat\delta Z_{\frac{(k-1)T}{2^n},\frac{kT}{2^n}}\big\|\geq T^\gamma 2^{-\gamma n}\Big)\leq 2^{p\gamma n}\sum_{k=1}^{2^n}\Expec{\big\|\hat\delta Z_{\frac{(k-1)T}{2^n},\frac{kT}{2^n}}\big\|^p}\lesssim T^{(\alpha-\gamma) p} 2^{n(p\gamma-p\alpha+1)}.
  \end{equation*}
  The right-hand side is summable in $n\in\N$, whence 
  \begin{equation}\label{eq:borel_cantelli}
    \max_{k=0,\dots,2^n}\big\|\hat\delta Z_{\frac{(k-1)T}{2^n},\frac{kT}{2^n}}\big\|\leq C \left(\frac{T}{2^{n}}\right)^\gamma\qquad\forall\,n\in\N
  \end{equation}
  for an almost surely finite random variable $C>0$ by a Borel-Cantelli argument.

  Let $s,t\in\bD$ with $s<t$. Then we can find an $n\in\N$ such that $|t-s|\in[2^{-(n+1)}T,2^{-n}T]$. Pick increasing sequences $(s_k)_{k\geq n}$ and $(t_k)_{k\geq n}$ converging to $s$ and $t$, respectively, such that
  \begin{itemize}
    \item $s_k,t_k\in\bD_k$ for each $k\geq n$,
    \item $t_n-s_n=2^{-n}T$, and
    \item $|s_{k+1}-s_k|,|t_{k+1}-t_k|\in\big\{0,2^{-(k+1)}T\big\}$ $\forall k\geq n$. 
  \end{itemize}
  By  the telescopic expansion of $Z_t$ using $Z_t=\lim_{k\to \infty} S_{t-t_k} Z_{t_k}$ and similarly for $Z_s$,  it follows from \eqref{eq:borel_cantelli} that 
  \begin{align*}
    \big\|\hat\delta Z_{s,t}\big\| &\leq\big\|S_{t-t_n}Z_{t_n}-S_{s-s_n}Z_{s_n}\big\| +\sum_{k=n}^\infty\big\|S_{s-s_{k+1}}\hat\delta Z_{s_k,s_{k+1}}\big\| +\sum_{k=n}^\infty\big\|S_{t-t_{k+1}}\hat\delta Z_{t_k,t_{k+1}}\big\| \\
    &\lesssim\big\|\hat\delta Z_{s_n,t_n}\big\| +\sum_{k=n}^\infty\big\|\hat\delta Z_{s_k,s_{k+1}}\big\| +\sum_{k=n}^\infty\big\|\hat\delta Z_{t_k,t_{k+1}}\big\| \\
    &\leq \frac{CT^\gamma}{2^{\gamma n}}+2CT^\gamma\sum_{k=n}^\infty \frac{1}{2^{\gamma (k+1)}}\leq 3C\left(\frac{T}{2^n}\right)^\gamma\leq 3\cdot 2^\gamma C|t-s|^\gamma.
  \end{align*}
  This gives the asserted H\"older continuity on the set $\bD$ and it is now standard to construct the desired modification of $(Z_t)_{t\in[0,T]}$.

  It remains to show that $\big\|\|Z\|_{\hat\C^\gamma}\big\|_{L^p}<\infty$. As usual, this follows from the Garsia-Rodemich-Rumsey inequality, c.f. Friz-Victoir  \cite[Theorem A.1]{Friz-Victoir}: 
   For each $\eta>\frac1p$, we have   \begin{equation*}
    \|f(t)-f(s)\| ^p\lesssim |t-s|^{\eta p-1}\int_s^t\int_s^t\frac{\|f(u)-f(v)\| ^p}{|v-u|^{\eta p+1}}\,du\,dv,
  \end{equation*}
  where the prefactor is independent of $f$. Choosing $\eta=\gamma+\frac1p$ and $f(r)=S_{t-r}Z_r$, we find
  \begin{align*}
    \Expec{\big\|Z\big\|_{\hat\C^\gamma}^p}&\lesssim\Expec{\sup_{s,t\in[0,T]}\int_s^t\int_s^t\frac{\big\|S_{t-u}Z_u-S_{t-v}Z_v\big\|^p}{|v-u|^{\gamma p+2}}}\\
    &\lesssim\int_0^T\int_0^v\frac{\Expec{\big\|\hat\delta Z_{u,v}\big\| ^p}}{|v-u|^{\gamma p+2}}\,du\,dv\lesssim\int_0^T\int_0^v |v-u|^{(\alpha-\gamma) p-2}\,du\,dv<\infty,
  \end{align*}
  as required.
\end{proof}

We also need the following technical result:
\begin{proposition}\label{prop:lp_bounds}
Let $p\ge 2$, $\delta\in(0,1)$, $\kappa\in \R$, $\gamma\in (0,1)$, and $\alpha\in \big(\f 12, 1\big)$. Let $\eta=H-\delta$ and $\bar\eta =\eta+\gamma\alpha$.
Let   $X\in  \hat{\B}_{\alpha,p}([0,T],\cH_{\kappa+\alpha})$ with $X_0\in L^p(\cH_{\kappa+\alpha})$. Suppose that 
 $$ 
  g: \Omega\times[0,T]\times\cH_{\kappa+\alpha}\to L(\cK,\cH_{\kappa+\alpha})
 $$ 
is independent of $\F_T^B$ and $\vertiii{g}_{\A^{p/\gamma}_{-\delta,\gamma}}<\infty$.  
 \begin{enumerate}
  \item  If $\eta>\f 12$ and $\bar \eta>1$, then the mixed Wiener-Young integral
  \begin{equation*}
    \int_s^t S_{t-r}g(r,X_r)\,\d B_r
  \end{equation*}
  defined in \eqref{eq:mixed_integral} is given by the $L^p$-limit of Riemann sums along an arbitrary sequence of partitions of $[s,t]$ with mesh tending to zero. 
\item Moreover, it satisfies the estimate
  \begin{equation}\label{eq:sewing_lp_bound}
    \left\|\int_s^t S_{t-r}g(r,X_r)\,\d B_r\right\|_{L^p({\cH_{\kappa+\alpha}})}\lesssim \vertiii{g}_{\A_{-\delta,\gamma}^{p/\gamma}(L(\cK,\cH_{\kappa+\alpha}))}\left(1+\|X\|_{\hat\B_{\alpha,p}}^\gamma\right)
   |t-s|^{\bar \eta -\alpha}.
  \end{equation}
 
 \item  If, in addition, $g\in\C^{(1-H)+,\gamma}$, then this integral coincides with the infinite-dimensional Young integral $\int_s^t S_{t-r}g(r,X_r)\,dB_r$, see \cref{prop:young}.
  \end{enumerate}
\end{proposition}

\begin{proof}
  Let 
  \begin{equation*}
    \Xi_{s,t}\define\int_s^t S_{t-s}g(r,X_s)\,\d B_r,\qquad 0\leq s\leq t\leq T.
  \end{equation*}
 We verify the requirements of \cref{prop:stochastic_mild_sewing}. By \cref{lem:lp_bound},
  \begin{equation*}
    \|\Xi_{s,t}\|_{L^p(\cH_{\kappa})}\lesssim \big\||S_{t-s}g(\bigcdot, X_s)|_{\C^{-\delta}([s,t],L(\cK,\cH_{\kappa}))}\big\|_{L^{p/\gamma}}|t-s|^{H-\delta}.
  \end{equation*}
  Since $\big\||S_{t-s}g(\bigcdot, X_s)|_{\C^{-\delta}([s,t],L(\cK,\cH_{\kappa}))}\big\|_{L^{p/\gamma}}\lesssim \vertiii{g}_{\A^{p/\gamma}_{-\delta,\gamma}(L(\cK,\cH_{\kappa}))}$ uniformly in $s,t\in[0,T]$, we conclude
  \begin{equation}\label{Xi-estimate}
    \|\Xi_{s,t}\|_{L^p(\cH_{\kappa})}\lesssim\vertiii{g}_{\A^{p/\gamma}_{-\delta,\gamma}(L(\cK,\cH_{\kappa}))}|t-s|^{\eta}.
  \end{equation}
  Let $\Lambda_{r,s,t}=\int_s^t \Big(S_{s-r}g(u,X_r)-g(u,X_s)\Big)\,\d B_u$. For the second order bound, we first notice that
  \begin{equation*}
    \hat\delta\Xi_{r,s,t}=-S_{t-s}\int_s^t \Big(g(u,X_s)-S_{s-r}g(u,X_r)\Big)\,\d B_u=S_{t-s}\Lambda_{r,s,t}.
  \end{equation*}
  \Cref{lem:lp_bound} yields
  \begin{equation*}
    \left\|\int_s^t \Big(g(u,X_s)-S_{s-r}g(u,X_r)\Big)\,\d B_u\right\|_{L^p(\cH_{\kappa})}\lesssim \Big\|\big|g(\bigcdot,X_s)-S_{s-r}g(\bigcdot,X_r)\big|_{\C^{-\delta}(L(\cK,\cH_{\kappa}))}\Big\|_{L^{p/\gamma}}|t-s|^{\eta}.
  \end{equation*}
  By the triangle inequality, we find
  \begin{align*}
    &\Big\|\big|g(\bigcdot,X_s)-S_{s-r}g(\bigcdot,X_r)\big|_{\C^{-\delta}(L(\cK,\cH_{\kappa}))}\Big\|_{L^{p/\gamma}}\\
    &\leq\Big\|\big|g(\bigcdot,X_s)-g(\bigcdot,S_{s-r}X_r)\big|_{\C^{-\delta}(L(\cK,\cH_{\kappa}))}\Big\|_{L^{p/\gamma}}
    +\Big\|\big|g(\bigcdot,S_{s-r}X_r)-g(\bigcdot,X_r)\big|_{\C^{-\delta}(L(\cK,\cH_{\kappa}))}\Big\|_{L^{p/\gamma}}\\
    &\phantom{\leq}+\Big\|\big|(\id-S_{s-r})g(\bigcdot,X_r)\big|_{\C^{-\delta}(L(\cK,\cH_{\kappa}))}\Big\|_{L^{p/\gamma}}.
      \end{align*}
For the first term, we have the bound
  \begin{align*}  \Big\|\big|g(\bigcdot,X_s)-g(\bigcdot,S_{s-r}X_r)\big|_{\C^{-\delta}(L(\cK,\cH_{\kappa}))}\Big\|_{L^{p/\gamma}} 
  &\leq  \vertiii{g}_{\A^{p/\gamma}_{-\delta,\gamma}(L(\cK,\cH_{\kappa+\alpha}))}\,
   \Big\| \|X_s-S_{s-r}X_r\|^\gamma_{\cH_{\kappa+\alpha}} \Big\|_{L^{p/\gamma}}
 \\
 &\lesssim \vertiii{g}_{\A^{p/\gamma}_{-\delta,\gamma}(L(\cK,\cH_{\kappa+\alpha}))}
  \|X\|^\gamma_{ \hat{\B}_{\alpha, p}([0,T], \cH_{\kappa+\alpha})}\;|s-r|^{\gamma\alpha}.
   \end{align*}
     Also,
 \begin{align*}  
 \Big\|\big|g(\bigcdot,S_{s-r}X_r)-g(\bigcdot,X_r)\big|_{\C^{-\delta}(L(\cK,\cH_{\kappa}))}\Big\|_{L^{p/\gamma}}
 &\leq  \vertiii{g}_{\A^{p/\gamma}_{-\delta,\gamma}(L(\cK,\cH_{\kappa+\alpha}))}
      \Big\| \|X_r-S_{s-r}X_r\|^\gamma_{\cH_{\kappa+\alpha}} \Big\|_{L^{p/\gamma}}\\
   &\lesssim    \vertiii{g}_{\A^{p/\gamma}_{-\delta,\gamma}(L(\cK,\cH_{\kappa+\alpha}))}
      \big\| \|X_r\|_{\cH_{\kappa}} \big\|_{L^p}^\gamma \;|s-r|^{\alpha}
   \end{align*}
The third term is bounded as below:
      \begin{align*}\Big\|\big|(\id-S_{s-r})g(\bigcdot,X_r)\big|_{\C^{-\kappa}(L(\cK,\cH_{\kappa}))}\Big\|_{L^{p/\gamma}}
 &  \leq |s-r|^\alpha \Big\|\big|g(\bigcdot,X_r)\big|_{\C^{-\kappa}(L(\cK,\cH_{\kappa+\alpha}))}\Big\|_{L^{p/\gamma}}\\
   &
   \leq    |s-r|^\alpha    \vertiii{g}_{\A^{p/\gamma}_{-\delta,\gamma}(L(\cK,\cH_{\kappa+\alpha}))}.
     \end{align*}
Our bounds show that
    \begin{align*}
    \big\|\Lambda_{r,s,t}\big\|_{L^p(\cH_\kappa)}
    \lesssim\vertiii{g}_{\A^{p/\gamma}_{-\delta,\gamma}}\left( \|X_r\|_{L^p(\cH_{\kappa+\alpha})}^\gamma+\|X\|_{\hat\B_{\alpha,p}(\cH_{\kappa})}^\gamma\right) |t-s|^\eta |s-r|^{\gamma\alpha}\\
       \lesssim\vertiii{g}_{\A^{p/\gamma}_{-\delta,\gamma}(L(\cK,\cH_{\kappa+\alpha}))}\left(1+\|X\|_{\hat\B_{\alpha,p}(\cH_{\kappa+\alpha})}^\gamma\right) |t-s|^{\bar\eta - 1} |t-r|.
  \end{align*}
 Here, we used that $\|X_r\|_{L^p(\cH_{\kappa+\alpha})}\lesssim 1+\|X\|_{\hat B_{\alpha,p}(\cH_{\kappa+\alpha})}$ and  $\gamma\alpha<1$ as well as $\eta>\bar \eta-1$. 
  In particular, we have also shown that
  \begin{equation*}
    \big\|\hat\delta\Xi_{r,s,t}\big\|_{L^p(\cH_{\kappa})}\lesssim\|g\|_{\A^{p/\gamma}_{-\delta,\gamma}}\left(1+\|X\|_{\hat\B_{\alpha,p}(\cH_{\kappa+\alpha})}^\gamma\right)|t-s|^\eta |s-r|^{\gamma\alpha},
  \end{equation*}
  whence $\Xi\in H_\eta^p(\cH_\kappa)\cap\bar H_{\bar\eta}^p(\cH_\kappa)$.   
The mild stochastic sewing lemma (\cref{prop:stochastic_mild_sewing}) applies. Hence, \eqref{Xi-estimate} yields
   \begin{align*}
  &  \left\|\int_s^t S_{t-r}g(r,X_r)\,\d B_r\right\|_{L^p({\cH_{\kappa+\alpha}})}\\
  &  \lesssim  \vertiii{g}_{\A^{p/\gamma}_{-\delta,\gamma}(L(\cK,\cH_{\kappa+\alpha}))}\left(1+\|X\|_{\hat\B_{\alpha,p}(\cH_{\kappa+\alpha})}^\gamma\right)|t-s|^{\bar \eta -\alpha}+
   \vertiii{g}_{\A^{p/\gamma}_{-\delta,\gamma}(L(\cK,\cH_{\kappa}))}|t-s|^{\eta}.
  \end{align*}
Since $\bar \eta-\alpha \le \eta$, the estimate \eqref{eq:sewing_lp_bound} follows.

 Let $g\in \C^{1-H+\epsilon, \gamma}$ for some $\epsilon>0$.  It remains to show that $\int_s^t S_{t-r}g(r,X_r)\,\d B_r$ coincides with the infinite-dimensional Young integral of \cref{prop:young}. But this is an easy consequence of the uniqueness statement of the mild stochastic sewing lemma and \cref{lem:lp_bound}:
  \begin{align*}
    \big\|\Xi_{u,v}-S_{v-u}g(u,X_u)B_{u,v}\big\|_{L^p(\cH_\kappa)}&\lesssim\left\|\int_u^v \Big(g(r,X_u)-g(u,X_u)\Big)\,dB_r\right\|_{L^p(\cH_\kappa)}\\
    &\lesssim\Big\|\sup_{r\in [u,v]}\big\|g(r,X_u)-g(u,X_u)\big\|_{L(\cK,\cH_\kappa)}\Big\|_{L^p} |v-u|^H\\
    &\lesssim |v-u|^{\rho}
  \end{align*}
  for some $\rho>1$.
\end{proof}

\begin{lemma}\label{lem:bound_drift}
  Let $X\in\hat\B_{\alpha,p}\big([0,T],\cH_{\kappa+\alpha}\big)$ and $ f: \Omega\times[0,T]\times\cH_{\kappa+\alpha}\to\cH_{\kappa+\alpha}$ is independent of  $\F_T^B$ with $\vertiii{f}_{\A_{-\delta,\gamma}^{p}}<\infty$, where $\kappa\in\R$, $\delta\in(0,1)$, and $\gamma(0,1]$. If $\delta<\gamma\alpha$, then
\begin{equation}\label{eq:lp_riemann}
    \left\|\int_s^t S_{t-r}f(r,X_r)\,dr\right\|_{L^p(\cH_{\kappa+\alpha})}\lesssim \vertiii{f}_{\A^{p/\gamma}_{-\delta,\gamma}(\cH_{\kappa+\alpha})}\Big(1+\|X\|_{\hat\B_{\alpha,p}(\cH)}^\gamma\Big) |t-s|^{1-\delta-(1-\gamma)\alpha}
  \end{equation}
  for each $0\leq s\leq t\leq T$.
\end{lemma}
\begin{proof}
  To bound the left-hand side, we employ \cref{prop:stochastic_mild_sewing} with
  \begin{equation*}
    \Xi_{s,t}\define S_{t-s}\int_s^t f(r,X_s)\,dr.
  \end{equation*}
  Clearly, for any $\gamma>0$, we have
  \begin{equation*}
    \|\Xi_{s,t}\|_{L^p(\cH_{\kappa})}\lesssim\vertiii{f}_{\A^{p/\gamma}_{-\delta,\gamma}(L(\cK,\cH_{\kappa+\alpha}))}|t-s|^{1-\delta}.
  \end{equation*}
  Since
  \begin{equation*}
    \hat\delta\Xi_{r,s,t}=-S_{t-s}\int_s^t f(u,X_s)-S_{s-r}f(u,X_r)\,du,
  \end{equation*}
  we have $\Lambda_{r,s,t}=\int_s^t \big(f(u,X_s)-S_{s-r}f(u,X_r)\big)\,du$. The rest of the proof follows closely that for \cref{prop:lp_bounds}: By definition, we have
  \begin{equation*}
      \left\|\int_s^t \Big(f(u,X_s)-S_{s-r}f(u,X_r)\Big)\,du\right\|_{L^p(\cH_{\kappa})}\lesssim \Big\|\big|f(\bigcdot,X_s)-S_{s-r}f(\bigcdot,X_r)\big|_{\C^{-\delta}(\cH_{\kappa})}\Big\|_{L^{p}}|t-s|^{1-\delta}.
  \end{equation*}
This leads to the estimate
  \begin{equation*}
    \big\|\Lambda_{r,s,t}\big\|_{L^p(\cH_{\kappa})}\lesssim \vertiii{f}_{\A^{p}_{-\delta,\gamma}(\cH_{\kappa+\alpha})}\Big(1+\|X\|_{\hat\B_{\alpha,p}(\cH_{\kappa+\alpha})}^\gamma\Big)|t-s|^{{\gamma\alpha-\delta}}|t-r|,
  \end{equation*}
  The mild stochastic sewing lemma (\cref{prop:stochastic_mild_sewing}) gives 
  \begin{align*}
  &\phantom{\lesssim}  \left\|\int_s^t S_{t-r}f(r,x_r)\,dr\right\|_{L^p(\cH_{\kappa+\alpha})}\\
    &\lesssim 
 \vertiii{f}_{\A^{p}_{-\delta,\gamma}(\cH_{\kappa+\alpha})}|t-s|^{1-\delta}
    + \vertiii{f}_{\A^{p}_{-\delta,\gamma}(\cH_{\kappa+\alpha})}\Big(1+\|X\|_{\hat\B_{\alpha,p}(\cH_{\kappa+\alpha})}^\gamma\Big)|t-s|^{1-\delta +\alpha \gamma-\alpha}\\
   & \lesssim \vertiii{f}_{\A^{p}_{-\delta,\gamma}(\cH_{\kappa+\alpha})}\Big(1+\|X\|_{\hat\B_{\alpha,p}(\cH_{\kappa+\alpha})}
   \Big)|t-s|^{1-\delta +\alpha \gamma-\alpha},
  \end{align*}
  as required. 
\end{proof}
 
These estimates apply to the solution of the slow SPDE \eqref{eq:slow_spde}. We record this in the following proposition:
 \begin{proposition}\label{non-uniform-estimate}
 Let  $T>0$, $H>\frac12$, and  $p\ge 2$. 
Assume the conditions of \cref{prop:well_posedness}. Then, for every initial condition $X_0\in \cH$ and every $T>0$, there exists a solution to
$
  dX_t=\big(A+f(t,X_t)\big)\,dt+g(t,X_t)\,dB_t$
 in $\hat\B_{\alpha,p}\big([0,T],\cH\big)$.
\end{proposition}
This is an immediate consequence of the deterministic a priori bound \eqref{eq:deterministic_apriori_bound} and the fact that, for each $\gamma<H$, $\|B\|_{\C^\gamma([0,T],\cK)}$ has moments of all order by virtue of Fernique's theorem. It is a good place to note that
SPDEs driven by fBMs have been widely studied, from the perspective of  rough paths, white noise analysis, Caputo-type derivatives, Malliavin calculus, and random dynamical systems, see \cite{Jentzen-Kloeden, Garrido2010, Caruana-Friz-Oberhauser, Duncan-Maslowski-Pasik-Duncan,Nualart-Vuillermot,Hernndez-Kolokoltsov-Toniazzi, Chen-Kim-Kim, Holden-Oksendal-Uboe-Zhang, Hesse-Neamtu, Brzezniak-Neerven-Salopek}. But our aim is to obtain  uniform bounds in terms of the rather weak $\vertiii{\bigcdot}_{\A_{-\delta, \gamma}^p}$-norm of the coefficients. As an application, we prove the convergence of the slow motion $X^\varepsilon$ \eqref{eq:slow_spde} to the effective equation \eqref{eq:averaged_spde}.

\begin{proposition} \label{uniform-Lp-bound}
  Let $\gamma \in (0,1)$, $\delta\in \big(0,H-\frac12\big)$, and $\alpha<\frac{H-\delta}{2-\gamma}$ satisfying $H-\delta+\alpha\gamma>1$.     Let  $f:\Omega\times[0,T]\times\cH\to\cH$ and  $g:\Omega\times[0,T]\times\cH\to L(\cK,\cH)$ be independent of $\F^B_\cdot$ with
       $$\vertiii{f}_{\A_{-\delta, \gamma}^p(\cH)}<\infty, \qquad \vertiii{g}_{\A_{-\delta, \gamma}^{p/\gamma}(L(\cK,\cH))}<\infty.$$
 Then the mild solution to the equation
 $$dX_t=A X_t dt+ f(t,X_t)dt+g(t, X_t)dB_t$$
 satisfies the a priori bound 
 \begin{equation}\label{eq:a_priori_lp }
\|X\|_{\hat\B_{\alpha,p}(\cH)}
\lesssim  \Big(1+\|X\|_{\hat\B_{\alpha,p}(\cH)}^\gamma\Big)
  \Big(\vertiii{f}_{\A_{-\delta,\gamma}^{p}(\cH)}+\vertiii{g}_{\A_{-\delta,\gamma}^{p/\gamma}(L(\cK,\cH))}\Big) |t-s|^{{H-\delta-\alpha(2-\gamma)}}. 
  \end{equation}
\end{proposition}

\begin{proof}
  We may apply the $L^p$-estimates obtained earlier in  \cref{prop:lp_bounds} and in \cref{lem:bound_drift} to the integrals in the identity
   $$\hat \delta X_{s,t}=\int_s^t S_{t-r} f(r,X_r)\,dr+\int_s^t S_{t-r}g(r,X_r)\,dB_r.$$
  In  \eqref{eq:sewing_lp_bound} and \eqref{eq:lp_riemann} take $\kappa=-\alpha$, then  
\begin{align*}
  \|\hat\delta X_{s,t}\|_{L^p}
  &\lesssim \Big(1+\|X\|_{\hat\B_{\alpha, p}(\cH)}^\gamma\Big)
  \Big(\vertiii{f}_{\A_{-\delta,\gamma}^{p}(\cH)}|t-s|^{1-\delta +\alpha \gamma-\alpha}+\vertiii{g}_{\A_{-\delta,\gamma}^{p}(L(\cK,\cH))}|t-s|^{H-\delta+\alpha\gamma-\alpha} \Big)  \end{align*}
   completing the proof. 
\end{proof}

\Cref{prop:lp_bounds,lem:bound_drift} yield the following probabilistic stability result for the SPDE \eqref{eq:spde}:
\begin{corollary}\label{cor:lp_stability}
 Let $H>\frac12$. Let $\delta\geq 0$, $\gamma<1$, and $\alpha>0$. Suppose that $$H-\delta>\frac12, \qquad \alpha<\frac{H-\delta}{2-\gamma}, \qquad H-\delta+\alpha\gamma>1.$$
   Let $f^n,f:\Omega\times[0,T]\times\cH\to\cH$ and $g^n,g:\Omega\times[0,T]\times\cH\to L(\cK,\cH)$ be independent of $\F^B_T$ and satisfy the conditions of \cref{prop:well_posedness}.
Then the SPDEs 
  \begin{align}
    dX_t^n&=\big(A+f^n(t,X_t^n)\big)\,dt+g^n(t,X_t^n)\,dB_t, \label{eq:xtn}\\
    dX_t&=\big(A+f(t,X_t)\big)\,dt+g(t,X_t)\,dB_t\label{eq:xt}
  \end{align}
  have a unique pathwise mild solution, see \eqref{def:mild_solution}. Moreover, if $X_0^n=X_0$ for each $n\in\N$ and 
  \begin{equation*}
    \bvertiii{f^n-f}_{\A_{-\delta,\gamma}^p(\cH)}\to 0,\qquad \bvertiii{g^n-g}_{\A_{-\delta,\gamma}^p(L(\cK,\cH))}\to 0,\qquad
  \end{equation*} then  $X^n\to\bar{X}$ in $\hat\C^{(\alpha-\frac1p)-}\big([0,T],\cH\big)$ in probability as $n\to\infty$.
\end{corollary}

\begin{proof}
  First note that the equations \eqref{eq:xtn}--\eqref{eq:xt} are well-posed in $\hat\C^{H-}\big([0,T],\cH\big)$ by \cref{prop:well_posedness}. \Cref{non-uniform-estimate} shows that $X^n\in \hat \B_{\alpha, p}([0,T],\cH)$ and \cref{uniform-Lp-bound} furnishes the estimate
  \begin{equation*}
\|X^n\|_{\hat\B_{\alpha,p}(\cH)}
\lesssim  \Big(1+\|X^n\|_{\hat\B_{\alpha,p}(\cH)}^\gamma\Big)
  \Big(\bvertiii{f^n}_{\A_{-\delta,\gamma}^{p}(\cH)}+\bvertiii{g^n}_{\A_{-\delta,\gamma}^{p/\gamma}(L(\cK,\cH))}\Big) |t-s|^{H-\delta-(2-\gamma)\alpha}.
  \end{equation*}
Since the exponent $H-\delta-(2-\gamma)\alpha>0$ by assumption, we can iterate this bound to see that $\sup_{n}  \|X^n\|_{\hat\B_{\alpha,p}(\cH)}<\infty$. 
  Next, we apply \cref{lem:residue} with 
  $y_t^n=X_0$ and
  \begin{equation*}
    \qquad\bar{y}_t^n\define X_0+\int_0^t S_{t-r}\big(f^n(r,X_r^n)-f(r,X_r^n)\big)\,dr+\int_0^t S_{t-r}\big(g^n(r,X_r^n)-g(r,X_r^n)\big)\,dB_r,
  \end{equation*}
  so that
  \begin{align*}
    X_t^n&=\bar y_t^n+\int_0^tS_{t-r} f(r,X_r^n)\,dr+\int_0^t S_{t-r}g(r,X_r^n)\,dB_r,\\
    X_t&=y_t^n+\int_0^tS_{t-r} f(r,X_r)\,dr+\int_0^t S_{t-r}g(r,X_r)\,dB_r.
  \end{align*}
  \Cref{uniform-Lp-bound} shows that
  \begin{equation*}
    \big\|y^n-\bar{y}^n\big\|_{\hat\B_{\alpha, p}(\cH)}\lesssim\Big(1+\|X^n\|_{\hat\B_{\alpha,p}(\cH)}^\gamma\Big)\Big(\bvertiii{f^n-f}_{\A_{-\delta,\gamma}^{p}(\cH)}+\bvertiii{g^n-g}_{\A_{-\delta,\gamma}^{p/\gamma}(L(\cK,\cH))}\Big),
  \end{equation*}
  whence the claim follows by \eqref{eq:residue_stability} and the mild Kolmogorov continuity theorem (\cref{prop:mild_kolmogorov}).
\end{proof}

\section{Non-Stationary Averaging}\label{sec:non_stat_averaging}
\label{sec-ergodic}

Let $\mu,\nu\in\P(\X)$. The \emph{total variation distance} between $\mu$ and $\nu$ is defined by
\begin{equation*}
  \TV{\mu-\nu}\define\sup_{A\in\B(\X)}|\mu(A)-\nu(A)|=\min_{(X,Y)\in\mathscr{C}(\mu,\nu)}\prob(X\neq Y),
\end{equation*}
where $\mathscr{C}(\mu,\nu)\define\{\rho\in\P(\X\times\X):\,[\rho]_1=\mu,\,[\rho]_2=\nu\}$ is the set of \emph{couplings}.

\subsection{Quantitative Ergodic Theorem}

In order to prove \cref{thm:averaging_spde}, we present a strengthened version of \cite[Lemma 3.14]{Hairer2020}. The proof there strongly relies on Kolmogorov's continuity theorem.   To the best of our knowledge, such a result is not known for random fields indexed by elements of a general infinite-dimensional Hilbert space. Here we follow a different strategy, which also allows us to extend the result to non-stationary fast processes.

Let $Y$ be a  stochastic process on a complete, separable metric space $\Y$ and $\pi\in\P(\Y)$ be a probability measure. Let $f: \cH\times\Y\to \cH$. Let 
$$F_\varepsilon(t,x) = f(x,Y_{\frac{t}{\varepsilon}}) \qquad  \hbox{and} \qquad \bar{f}(x)\define\int_{\Y}f(x,y)\pi(dy).$$
We also set $\F_t^Y=\sigma(Y_s,\,s\leq t)$.

\begin{definition}\label{def-ergodic}
We say that the \emph{ergodic condition} with rate $a>0$ and exponent $\delta\in(0,1)$ holds for $f$  (w.r.t. $Y$ and $\pi$) if for any $\cH$-valued random variables $X,Y$,
\begin{align*}
  \sup_{s\leq T}\sup_{X,Z\in\F_s^B\vee\F^Y_{\frac{s}{\varepsilon}}}\frac{\big\|\big|F_\varepsilon(\bigcdot,X)-F_\varepsilon(\bigcdot,Z)-\big(\bar f(X)-\bar f(Z)\big)\big|_{\C^{-\delta}([s,T])}\big\|_{L^p}}{\|X-Z\|_{L^p}}&\lesssim \varepsilon^a, \\
  \sup_{s\leq T}\sup_{X\in\F_s^B\vee\F^Y_{\frac{s}{\varepsilon}}}\Big\|\big|F_\varepsilon(\bigcdot,X)-\bar f(X)\big|_{\C^{-\delta}([s,T])}\Big\|_{L^p}&\lesssim \varepsilon^a.
\end{align*}
The constants are allowed to depend on $f$, but \emph{not} on $\varepsilon>0$. In particular,
\begin{equation*}
  \bvertiii{F_\varepsilon-\bar{f}}_{\A_{-\delta,\gamma}^p}\lesssim\varepsilon^a
\end{equation*}
for the norm \eqref{eq:a_norm}.
\end{definition}

We let $\L(Y_{t}\,|\,Y_s)$ denote the conditional law of $Y_t$ given $Y_s$. Let $\bar{B}_R$ denote the closed ball in $\cH$ of radius $R$ centered at $0$. For simplicity, we formulate the next lemma only for Markovian fast processes. However, we shall show in \cref{sec:non_markov} that its conclusion also holds for some non-Markovian process~$Y$.

\begin{lemma}\label{lem:norm_convergence}
Let $(Y_t)_{t\geq 0}$ be a Markov process on $\CY$. Let $\pi$ be a probability measure on $\CY$. Assume that there is a number $\rho\in(0,1)$ such that, for all $0\leq s\leq t$,
  \begin{equation}\label{eq:tv_decay_assumption}
   \Expec{\TV{\L(Y_{t}\,|\,Y_s)-\pi}}\lesssim \f{ 1  }{(t-s)^{\rho}}.
  \end{equation}
  Let $f:\cH\times\Y\to\cH$ and  $g:\cH\times\Y\to L(\cK,\cH)$ be  bounded measurable functions for which  there is an $L>0$ such that for all 
  $x,z\in\cH$ and all $y\in\Y$:
  \begin{equation}\label{eq:assumptions_f}
      \big\|f(x,y)-f(z,y)\big\|_{\cH}\leq L\|x-z\|_\cH, \qquad
       \big\|g(x,y)-g(z,y)\big\|_{L(\cK,\cH)}\leq L\|x-z\|_\cH.
  \end{equation}
  Fix  $\delta\in(0,1)$ and $\gamma\in(0,1)$. Then, for every $p>\frac{1+\rho}{\kappa}\vee 2$, the ergodic condition with rate $\f \rho p$ and exponent $\delta$ holds for both $f$ and $g$.
\end{lemma}

\begin{proof} 
  We only prove the statement for $g$; for $f$ one proceeds \emph{mutatis mutandis}.  Since $\bar{g}$ is bounded and Lipschitz continuous, we can subtract it from $g$ so that we may assume $\bar{g}\equiv 0$ without any loss of generality. Let $e\in\cK$ be a unit vector. 

  Let $s\in[0,T]$ and fix random variables $X,Z\in\F_s^B\vee\F^Y_{\frac{s}{\varepsilon}}$. For any $x,z\in\cH$, we abbreviate
    \begin{equation*}
      \fG_{s,t}^\varepsilon(x,z)=\int_s^t \Big(G_\varepsilon(u,x)e-G_\varepsilon(u,z)e\Big)\,du,
    \end{equation*}
 which is a random variable with values in $\cH$, and  observe the trivial bound 
    \begin{equation*}
      \big\|\fG_{s,t}^\varepsilon(x,z)\big\|_{L^\infty(\Omega)}^{p-2}=\left\|\int_s^t \Big(G_\varepsilon(u,x)e-G_\varepsilon(u,z)e\Big)\,du\right\|_{L^\infty(\Omega)}^{p-2}\lesssim \|x-z\|^{p-2} |t-s|^{p-2}.
    \end{equation*} 
By Markovity,  $Y_{\frac{u}{\varepsilon}}$ is conditionally independent of $\F^B_s\vee\F^Y_{\frac{s}{\varepsilon}}$ given $Y_{\frac{s}{\varepsilon}}$, so we have
    \begin{align}
    &  \Expec{\left\|\int_s^t \Big(G_\varepsilon(u,X)e-G_\varepsilon(u,Z)e\Big)\,du\right\|_{\cH}^p}=\Expec{\Expec{\big\|\fG_{s,t}^\varepsilon(x,z)\big\|_{\cH}^p\,\middle|\,Y_{\frac{s}{\varepsilon}}}\Big|_{x=X,z=Z}} \nonumber\\
      &\leq\Expec{\Expec{\big\|\fG_{s,t}^\varepsilon(x,z)\big\|_\cH^2\,\middle|\,Y_{\frac{s}{\varepsilon}}}\big\|\fG_{s,t}^\varepsilon(x,z)\big\|_{L^\infty(\Omega)}^{p-2}\Big|_{x=X,z=Z}}. \label{eq:split_independence}
    \end{align}
    The first factor can be expanded as
    \begin{align*}
        &\Expec{\big\|\fG_{s,t}^\varepsilon(x,z)\big\|^2\,\middle|\,Y_{\frac{s}{\varepsilon}}}=\Expec{\bigg\|\int_s^t \Big(G_\varepsilon(u,x)e-G_\varepsilon(u,z)e\Big)\,du\bigg\|^2\,\bigg|\,Y_{\frac{s}{\varepsilon}}}\\
        =&2\int_s^t\int_s^v\Expec{\Braket{g\big(x,Y_{\frac{u}{\varepsilon}} \big)e-g\big(z,Y_{\frac{u}{\varepsilon}} \big)e,\Expec{g\big(x,Y_{\frac{v}{\varepsilon}} \big)e-g\big(z,Y_{\frac{v}{\varepsilon}} \big)e\,\middle|\,Y_{\frac{u}{\varepsilon}} }}\,\middle|\,Y_{\frac{s}{\varepsilon}}}\bigg|_{x=X,z=Z}\,du\,dv.
    \end{align*}
    We notice that, by \eqref{eq:tv_decay_assumption}--\eqref{eq:assumptions_f} and the assumption $\bar{g}\equiv 0$, 
    \begin{align*}
        \Big\|\Expec{g\big(x,Y_{\frac{v}{\varepsilon}} \big)e-g\big(z,Y_{\frac{v}{\varepsilon}} \big)e\,\middle|\,Y_{\frac{u}{\varepsilon}} }\Big\|
        &\leq \big\|g(x,\bigcdot)e-g(z,\bigcdot)e\big\|_\infty\TV{\L(Y_{\frac{v}{\varepsilon}} \,|\,Y_{\frac{u}{\varepsilon}} )-\pi}\\
        &\lesssim  \|x-z\| \left(\frac{\varepsilon}{v-u}\right)^{\rho}.
    \end{align*}
    Hence by Cauchy-Schwarz,
    \begin{equation*}
        \Expec{\big\|\fG_{s,t}^\varepsilon(x,z)\big\|^2\,\middle|\,Y_{\frac{s}{\varepsilon}}}
        \lesssim \|x-z\|^2\int_s^t\int_s^v\left(\frac{\varepsilon}{v-u}\right)^\rho\,du\,dv 
        \lesssim \varepsilon^\rho\|x-z\|^2|t-s|^{2-\rho}. 
    \end{equation*}
       in view of \eqref{eq:assumptions_f}. Inserting this back into \eqref{eq:split_independence}, we have shown that
    \begin{equation*}
      \left\|\int_s^t G_\varepsilon(u,X)e-G_\varepsilon(u,Z)e\,du\right\|_{L^p}\lesssim\varepsilon^{\frac\rho p}|t-s|^{1-\frac{\rho}{p}}\|X-Z\|_{L^p}.
    \end{equation*}
    Since this bound is uniform in both $0\leq s<t\leq T$ and $e\in\{x\in\cK:\,\|x\|_\cK=1\}$, Kolmogorov's continuity theorem implies
    \begin{equation*}
        \Big\|\big|G_\varepsilon(\bigcdot,X)-G_\varepsilon(\bigcdot,Z)-\big(\bar g(X)-\bar g(Z)\big)\big|_{\C^{-\delta}(L(\cK,\cH))}\Big\|_{L^p}\lesssim \varepsilon^{\frac\rho p}\|X-Z\|_{L^p},
    \end{equation*}
    provided that we choose $p>\frac{1+\rho}{\delta}$.

    We also have 
    \begin{equation*}
        \Big\|\Expec{g\big(x,Y_{\frac{v}{\varepsilon}} \big)e\,\middle|\,Y_{\frac{u}{\varepsilon}} }\Big\|\leq \big\|g(x,\bigcdot)e\big\|\TV{\L(Y_{\frac{v}{\varepsilon}} \,|\,Y_{\frac{u}{\varepsilon}} )-\pi}
        \lesssim  \; \|g\|_\infty \left(\frac{\varepsilon}{v-u}\right)^{\rho}.
    \end{equation*}
Hence replacing $G_\varepsilon(u,x)e-G_\varepsilon(u,z)e$ by $G_\varepsilon(u,x)e$ in the definition of $\fG_{s,t}^\varepsilon(x,z)$ 
leads to 
\begin{equation*}
 \Big\|\big|G_\varepsilon(\bigcdot,X)-\bar g(X)\big|_{\C^{-\delta}(L(\cK,\cH))}\Big\|_{L^p}\lesssim \varepsilon^{\f \rho p}.
\end{equation*}
 This completes the proof.
\end{proof}

%
%
%
%

\subsection{Examples}

A class of fast process falling in the regime of \cref{lem:norm_convergence} are solutions of S(P)DEs  driven by another \emph{fractional} Brownian motion. Establishing the convergence of the conditional law is much more delicate in this case and there not many results known.  For a (time-inhomogeneous) Markov process with transition kernel $(\sP_{s,t})_{0\leq s\leq t}$, the assumption \eqref{eq:tv_decay_assumption} is implied by
 $ \TV{\sP_{s,t}(y,\cdot)-\pi}\leq \frac{C(y)}{(t-s)^\rho}$ for all $0\leq s\leq t$ and $y\in\Y$,
provided that $\sup_{s\geq 0}\Expec{C(Y_s)}<\infty$. Such a bound is well known to hold for solutions of SDEs satisfying strong H\"ormander's conditions on compact manifolds.  There are of course many examples of random dynamical systems with rate of convergence to invariant measure at least polynomial and for which the above bound holds, we shall focus on the less exploited non-Markovian dynamics.

\subsubsection{A Non-Markovian Example}\label{sec:non_markov}






The following \emph{quenched ergodic theorem} furnishes non-Markovian examples falling in the regime of the averaging principle in the next section:
 Consider the SDE
  \begin{equation}\label{eq:sde}
    dY_t=b(Y_t)\,dt+\sigma\,d\hat{B}_t,
  \end{equation}
  where $\sigma\in\Lin{n}$ is invertible and $\hat{B}$ is an $n$-dimensional fBm independent of $B$. 

  To make sense of an invariant measure for the equation \eqref{eq:sde}, we course cannot resort to the notion of an invariant measure for a Markov process. A natural replacement can be found in the theory of random dynamical systems. However, the latter invariant measures are not guaranteed to be physically meaningful since the solution may `look' into the future. Hence, Hairer \cite{Hairer2005} introduced the notion of a `physical invariant measure' for which a Markov process is defined by including the history of the process---instead of the whole two-sided noise process as used in the theory of random dynamical system. The history space $\H_H$ is a separable Banach space defined as the closure of the space 
\begin{equation*}
  \C_0^\infty\big((-\infty,0],\R^n\big)\define\{w: (-\infty, 0]\to \R^n \text{ smooth with compact support and } w(0)=0\},
\end{equation*}
under the norm $
  \|w\|_H\define\sup_{s,t\le 0} \f {|w(t)-w(s)|}{|t-s|^{\f 12(1-H)}(1+|t|+|s|)^{\f 12}}$. This space supports the Wiener measure.
 The construction of the invariant measure then comes down to constructing a family of Feller transition probabilities $(\mathscr{Q}_t)_{t\in[0,T]}$ on $\R^n\times \H_H$. Then let $(Y_t^T, Z_t^T)_{t\in[0,T]}$ be a Markov process on $\R^n\times\H_H$ with transition probabilities $(Q_t)$ and started from an invariant measure whose projection to $\H_H$ is the Wiener measure. One can then take $T\to \infty$ to obtain a weak limit $(Y_t, Z_t)_{t\geq 0}$. The first marginal of this Feller process solves \eqref{eq:sde} and is stationary. Its time $0$ marginal $\pi\in\P(\R^n)$ is used in the next proposition:

\begin{proposition}\label{prop:sde}
  Suppose that $b:\R^n\to\R^n$ is globally Lipschitz continuous and there are $\lambda,\kappa,R>0$ such that 
  \begin{equation*}
     \braket{b(x)-b(y),x-y}\leq\begin{cases}
      -\kappa|x-y|^2, &|x|,|y|> R,\\
      \lambda|x-y|^2, &\text{otherwise}.
     \end{cases}
   \end{equation*} 
   Then there is a $\Lambda=\Lambda(\kappa, R)>0$ such that, provided $\lambda\leq\Lambda$, there exists a physical invariant probability measure $\pi\in\P(\R^n)$ to \eqref{eq:sde} and, for any initial condition $Y_0\in L^1$, the conclusion of \cref{lem:norm_convergence} holds:

  Let $f:\cH\times\R^n\to\cH$ and  $g:\cH\times\R^n\to L(\cK,\cH)$ be  bounded measurable functions for which  there is an $L>0$ such that for all $x,z\in\cH$ and all $y\in\R^n$:
    \begin{equation*}
        \big\|f(x,y)-f(z,y)\big\|_{\cH}\leq L\|x-z\|_\cH, \qquad
         \big\|g(x,y)-g(z,y)\big\|_{L(\cK,\cH)}\leq L\|x-z\|_\cH.
    \end{equation*}
    Fix  $\delta\in(0,1)$ and $\gamma\in(0,1)$. Then, for every $p>\frac{1+\rho}{\kappa}\vee 2$, there is a $c>0$ such that the ergodic condition with exponential rate $e^{-\frac{c}{p}t}$ and exponent $\delta$ holds for both $f$ and $g$.
\end{proposition}

\begin{proof}
  The only difference to \cref{lem:norm_convergence} lies in the conditioning step \eqref{eq:split_independence}: Let $\Phi_{s,t}\big(y,\hat B)$ denote the flow to \eqref{eq:sde} interpreted in the pathwise sense. We can decompose $\hat B=\bar{\hat{B}}^{\frac{s}{\varepsilon}}+\tilde{\hat{B}}^{\frac{s}{\varepsilon}}$ as in \eqref{eq:increment_decomposition}. Since $\tilde{\hat{B}}^{\frac{s}{\varepsilon}}$ is independent of $\F_s^B\vee\F^Y_{\frac{s}{\varepsilon}}$, \eqref{eq:split_independence} becomes
  \begin{equation*}
      \Expec{\left\|\int_s^t \Big(G_\varepsilon(u,X)e-G_\varepsilon(u,Z)e\Big)\,du\right\|_{\cH}^p}\leq\Expec{\Expec{\big\|\fG_{s,t}^\varepsilon(x,z)\big\|_\cH^2\,\middle|\,\F^B_{\frac{s}{\varepsilon}}}\big\|\fG_{s,t}^\varepsilon(x,z)\big\|_{L^\infty(\Omega)}^{p-2}\Big|_{x=X,z=Z}},
    \end{equation*}
    see \cite[Lemma 3.6]{Li2020} for details. The \emph{quenched ergodic theorem} of \cite[Corollary 3.24]{Li2020} shows that
    \begin{equation*}
      \Expec{\TV{\L(Y_t\,|\,(Y_r)_{r\leq s}}}\lesssim e^{-c(t-s)},
    \end{equation*}
    whence we can conclude the proof along the same lines as in \cref{lem:norm_convergence}.
\end{proof}

\subsubsection{The Stationary Process Case}

Let us draw a connection of our results to \cite{Hairer2020}. There, the authors considered a \emph{finite-dimensional} slow SDE and proved a similar averaging principle for time $t$ mixing, stationary fast processes. This is to say that Rosenblatt's mixing coefficient
\begin{equation*}
  \alpha(t)=\alpha(Y_0,Y_t)=\sup_{\substack{A\in\sigma(Y_0)\\B\in\sigma(Y_t)}}\big|\prob(A\cap B)-\prob(A)\prob(B)\big|
\end{equation*}
vanishes with an algebraic rate as $t\to\infty$. Let $\pi\in\P(\Y)$ be the stationary measure of $Y$.  If  $  \TV{\L(Y_t\,|\,Y_0=y)-\pi}\leq C(y)\beta(t)$ for $\pi$-a.e. $y\in\Y$, then we have that $ \alpha(t)\leq\|C\|_{L^1(\Y,\pi)}\beta(t)$.
Indeed, fix $A,B\in\B(\Y)$ and $t\geq 0$. On disintegration we have
\begin{align*}
  &\phantom{=}|\prob(Y_0\in A,Y_t\in B)-\prob(Y_0\in A)\prob(Y_t\in B)|\\
  &= \left| \int_{A}\Big(\L(Y_t\,|\,Y_0=y)(B)-\pi(B)\Big)\,\pi(dy)\right|\leq \beta(t)\int_{\Y} C(y)\,\pi(dy)=\|C\|_{L^1(\Y,\pi)}\beta(t).
\end{align*} 
We emphasize that we cannot cover such a general class of fast process. This is due to the failure of Kolmogorov's continuity theorem for random fields indexed by a general infinite-dimensional Hilbert space, which made the more challenging conditioning procedure necessary. In particular, the identity \eqref{eq:split_independence} may not hold for general $Y$. However, the classes of fast processes presented in \cref{sec:non_markov} cover a very rich selection of practically relevant examples.

\section{Application to Slow-Fast Systems}\label{sec:slow_fast}

Recall that $(Y_t)_{t\geq 0}$ is a sample continuous process on a separable complete metric space  $(\Y, d_{Y})$, which is independent of the fBm $B$. We assume that the ergodic condition introduced in \cref{def-ergodic} holds. Examples of $Y$ satisfying this assumption are given in the last section.

We can now prove a fractional averaging principle for the two-scale dynamics \eqref{eq:slow_spde}. To this end, let us first state the regularity assumption on the coefficients:
\begin{condition}\label{cond:coefficients}
  The vector fields $f:\cH\times\Y\to\cH$, $g:\cH\times\Y\to L(\cK,\cH)$ are globally Lipschitz continuous. Moreover, for each $\alpha<1-H$,
  \begin{alignat*}{4}
    \sup_{y\in\Y}  \big\|g(x,y)-g(z,y)\big\|_{L(\cK,\cH_{-\alpha})} &\lesssim \|x-z\|_{\cH_{-\alpha}} &\qquad \forall\,&x,z\in\cH,\\
     \sup_{y\in\Y}  \|D_x g(x,y)-D_xg(z,y)\|_{L(\cK,L(\cH_{-\alpha}))}&\lesssim \|x-z\|_{\cH_{-\alpha}}  &\qquad \forall\,&x,z\in\cH,\\
    \sup_{x\in \cH}  \|D_x g(x,y)-D_xg(x,\bar{y})\|_{L(\cK,L(\cH_{-\alpha}))}&\lesssim d_\Y(y,\bar{y})  &\qquad \forall\,&y,\bar{y}\in\Y,\\
    \sup_{x\in \cH}   \big\|g(x,y)-g(x,\bar{y})\big\|_{L(\cK,\cH_{-\alpha})}&\lesssim d_\Y(y,\bar{y}) &\qquad \forall\,&y,\bar{y}\in\Y.
    \end{alignat*}
\end{condition}
\begin{remark}
  Let $f,g$ satisfy \cref{cond:coefficients} and define 
  $$ f_\varepsilon(t,x)=f(x, Y_{\f t \varepsilon}), \qquad \qquad 
  g_\varepsilon(t,x)=g(x,  Y_{\f t \varepsilon}).$$
  Then $f_\varepsilon, g_\varepsilon$ satisfies the conditions of Proposition \ref{prop:well_posedness} and therefore for each $\varepsilon>0$, the SPDE \Cref{eq:slow_spde} has a unique solution in $\hat\B_{\alpha,p}([0,T],\cH)$ for any $T>0$, any $p\ge 2$, any $\alpha<H$, and any $X_0\in \cH$, see \cref{non-uniform-estimate}.
\end{remark}
\begin{theorem}\label{thm:averaging_spde}
Let $Y$ be a stochastic process with almost surely locally H\"older continuous sample paths of order greater than  $1-H$. Suppose that either $Y$ falls in the regime of \cref{lem:norm_convergence} or \cref{prop:sde}. If $f$ and $g$ satisfy \cref{cond:coefficients} and $X_0\in\bigcap_{p\geq 1}L^p(\Omega,\cH)$, then the following hold for any $\alpha<H$:
  \begin{enumerate}
    \item As $\varepsilon\to 0$, the unique solution $X^\varepsilon\in \hat\C^\alpha\big([0,T],\cH\big)$ to \eqref{eq:slow_spde} converges in probability in $\hat\C^\alpha\big([0,T],\cH\big)$.
    \item The limit is the unique pathwise solution to \eqref{eq:averaged_spde}.
  \end{enumerate}
\end{theorem}

\begin{remark}
  Since we only made use of the covariance structure and hypercontractivity estimates, we notice that the results of \cref{sec:lp_estimates} and consequently \cref{thm:averaging_spde} hold \emph{mutatis mutandis} for suitable non-Gaussian processes like for example an infinite-dimensional version of the Rosenblatt process \cite{Rosenblatt1961,Taqqu1974,Tudor2008}.
\end{remark}

\begin{proof}[Proof of \cref{thm:averaging_spde}]
  Fix $\alpha\in\big(\frac12,H\big)$. Let $R>0$ and let $\varrho_R\in\C^\infty_b(\cH)$ be a smooth bump function with value $1$ on $\bar{B}_R$ and which vanishes on $\cH\setminus \bar B_{2R}$. Since the map $x\mapsto\|x\|^2$ is $\C^\infty$ in the Fr\'echet sense, such a function can be constructed. We also define the stopping time $$\tau_R^\varepsilon=\inf\{t\geq 0:\,\|X_t^\varepsilon\|\vee\|\bar{X}\|>R\}.$$
  Set $f_R(t,x)=\varrho_R(x) f(t,x)$ and similar for $g_R$. Notice that $f_R$ and $g_R$ are bounded by the Lipschitz assumption. Let
  \begin{equation*}
    dX_t^{\varepsilon,R}=\big(AX_t^{\varepsilon,R}+f_R(X_t^{\varepsilon,R},Y_t^\varepsilon)\big)\,dt + g_R(X_t^{\varepsilon,R},Y_t^\varepsilon)\,dB_t
  \end{equation*}
  and
  \begin{equation*}
    d\bar X_t^{R}=\big(A\bar X_t^{R}+\bar{f}_R(\bar X_t^{R}\big)\,dt\big) + \bar{g}_R(\bar X_t^{R})\,dB_t.
  \end{equation*}
  Note that both of these equations are well posed in $\hat{\C}^\alpha([0,T],\cH)$ by \cref{prop:well_posedness}.

  Fix $\delta>0$ and $\gamma<1$ such that $H-\delta+\alpha\gamma>1$ and $\alpha<\frac{H-\delta}{2-\gamma}$. Let $c_1\in(0,1)$ be given.
     Since $\|\hat\delta f_{0,t}\|\ge \|f_t\|-\|S_tf_0\|$, we can estimate
    \begin{align*}
    \prob\left(\big\|X^\varepsilon\big\|_{\C([0,\tau_R^\varepsilon],\cH)}\geq R\right)&\leq \prob\left(\big\|X^\varepsilon\big\|_{\hat\C^\alpha([0,\tau_R^\varepsilon],\cH)}\geq\frac{R-\|S_tX_0\|_\cH}{T^\alpha}\right)\\
    &\leq \prob\left(\big\|X^\varepsilon\big\|_{\hat\C^\alpha([0,\tau_R^\varepsilon],\cH)}\geq\frac{R-M}{T^\alpha}\right) + \prob\big(\|S_tX_0\|_\cH\geq M\big)\\
    &\leq \prob\left(\big\|X^{\varepsilon, R}-\bar X^R\big\|_{\hat\C^\alpha([0,T],\cH)}\geq c_1\right)+\prob\left(\|\bar X\|_{\hat\C^\alpha([0,T],\cH)}\geq \frac{R-M}{T^\alpha}-c_1 \right)\\
    &\phantom{\leq} + \prob\big(\|S_tX_0\|_\cH\geq   M \big).
  \end{align*}
  Hence, we have
  \begin{align*}
    \prob\big(\tau_R^\varepsilon<T\big)&\leq \prob\left(\big\|X^\varepsilon\big\|_{\C([0,\tau_R^\varepsilon],\cH)}\geq R\right)+\prob\left(\big\|\bar{X}\big\|_{\C([0,\tau_R^\varepsilon],\cH)}\geq R\right) \\
    &\leq\prob\left(\big\|X^{\varepsilon, R}-\bar X^R\big\|_{\hat\C^\alpha([0,T],\cH)}\geq c_1\right)+2\prob\left(\|\bar X\|_{\hat\C^\alpha([0,T],\cH)}\geq \frac{R-M}{T^\alpha}-c_1 \right)\\
    &\phantom{\leq} + 2\prob\big(\|S_tX_0\|_\cH\geq   M \big)
  \end{align*}
Let $c_2>0$.  Since $X_0<\infty$ with probability $1$, we can first choose $M_0>0$ such that
  \begin{equation*}
    \prob\big(\|S_t X_0\|_\cH\geq M_0\big)\leq\frac{ c_2}{8},
  \end{equation*}
Since $\bar f, \bar g$ satisfies conditions of \cref{prop:well_posedness},
\cref{uniform-Lp-bound} applies yielding: For each $p\geq 1$,
 $\bar X\in L^p\big(\Omega,\hat\C^\alpha([0,T],\cH)\big)\subset \hat\B_{\alpha,p}$. Hence, pick $R_0>0$ such that
  \begin{equation*}
    \prob\left(\|\bar X\|_{\hat\C^\alpha([0,T],\cH)}\geq \frac{R_0-M_0}{T^\alpha}-c_1 \right)\leq\frac{ c_2}{8}.
  \end{equation*}
 Then we have that
  \begin{equation}\label{eq:convergence_prob}
    \prob\left(\big\|X^\varepsilon-\bar{X}\big\|_{\hat\C^\alpha([0,T],\cH)}> c_1\right)\leq\prob(\tau_{R_0}^\varepsilon< T) + \prob\left(\big\|X^{\varepsilon,R_0}-\bar{X}^{R_0}\big\|_{\hat\C^\alpha([0,T],\cH)}> c_1\right).
  \end{equation}
  Using that $\supp\big(f_{R_0}(t,\bigcdot)\big)\subset\bar{B}_{R_0}$ for each $t\in[0,T]$, we have
  \begin{equation*}
    \bvertiii{f^\varepsilon_{R_0}-\bar{f}_{R_0}}_{\A_{-\delta,\gamma}^p}
    \leq\bvertiii{f^\varepsilon_{R_0}-\bar{f}_{R_0}}_{\A_{-\delta,1}^p}
  \end{equation*}
  and similar for $g_{R_0}$. By \cref{lem:norm_convergence} (resp. \cref{prop:sde}), it holds that
  \begin{equation*}
  \bvertiii{f^\varepsilon_{R_0}-\bar{f}_{R_0}}_{\A_{-\delta,1}^p}\lesssim \varepsilon^{\f \rho p},\qquad \bvertiii{g^\varepsilon_{R_0}-\bar{g}_{R_0}}_{\A_{-\delta,1}^p}\lesssim \varepsilon^{\f \rho p}.
\end{equation*}
Hence by \cref{cor:lp_stability} on the stability of solutions, we can choose $\varepsilon_0>0$ so that, for any $\varepsilon<\varepsilon_0$, 
  \begin{equation*}
    \prob\left(\big\|X^{\varepsilon,R_0}-\bar{X}^{R_0}\big\|_{\hat\C^\alpha([0,T],\cH)}> c_1\right)\leq\frac{ c_2}{4}.
  \end{equation*}
  Coming back to \eqref{eq:convergence_prob}, we have shown that, for any $\varepsilon\in (0,\varepsilon_0)$,
  \begin{align*}
    &\prob\left(\big\|X^\varepsilon-\bar{X}\big\|_{\hat\C^\alpha([0,T],\cH)}> c_1\right)
    \leq 2\prob\left(\big\|X^{\varepsilon,R_0}-\bar{X}^{R_0}\big\|_{\hat\C^\alpha([0,T],\cH)}> c_1\right) +\f { c_2} 2 \le  c_2  \end{align*}
  as required.
\end{proof}

For slow-fast SPDEs driven by Wiener processes where $g(x,y)=g(y)$,  it is known that $X^\varepsilon\to\bar{X}$  in $L^p\big(\Omega,\C([0,T],\cH)\big)$  \cite{Liu2019}. [This is a common feature with SDEs.] To obtain such strong convergence for Markovian systems, the requirement $g(x,y)=g(x)$ is necessary even for SDEs as the following very simple example shows:
\begin{example}
Let $\tilde W$ and $W$ be independent one-dimensional Wiener processes.  Consider  $dX_t^\varepsilon=\cos(Y_{\frac{t}{\varepsilon}})\,d\tilde W_t$ and $dY_t=-Y_t\,dt+\sqrt{2}\,dW_t$ with $Y_0\sim \pi\define N(0,1)$ and $X_0^\varepsilon=X_0$.

Let $\eta\define\sqrt{\int_\R \cos^2(y)\,\pi(dy)}=\sqrt{\frac12\big(1+e^{-2}\big)}$. Then $X^\varepsilon\Rightarrow \bar X=\eta \tilde W$ weakly in $\C\big([0,T],\R\big)$, but
  \begin{equation*}
    \|X^\varepsilon_t-\bar{X}_t\|_{L^2}^2=\int_0^t\Expec{\big(\cos(Y_{\frac{s}{\varepsilon}})-\eta\big)^2}\,ds=\left(1+\frac{1-\sqrt{2(e+e^3)}}{e^2}\right)t \not \to 0.
  \end{equation*}
\end{example}
In our case, the convergence is in probability, even though $g$ depend on both the slow and the fast variable, this is a feature of fractional averaging.

 {
\footnotesize
\bibliographystyle{plain}
\bibliography{./SD-D-21-00139}
}

\end{document}